\pdfoutput=1
\documentclass[
12pt,
a4paper,
final,
disable,
]{article}

\usepackage[
]{artmacs}
\usepackage{filecontents}

\definecolor{light-blue}{rgb}{0.8,0.85,1}
\definecolor{light-grey}{gray}{0.75}

\usepackage{sagetex}
\usepackage{rotating}

\newcommand{\intset}[1]{{#1}}

\newcommand{\divisors}{\intset{N}}

\newcommand{\relset}[1]{{\mathsf #1}}
\renewcommand{\vec}[1]{\relset{#1}}
\newcommand{\uset}[1]{\underline{#1}}

\newcommand{\udg}[1]{\overline{#1}}    
\renewcommand{\dag}[1]{\overrightarrow{#1}}    

\newcommand{\polyset}[1]{{\mathcal #1}}
\newcommand{\pP}[1]{\polyset{P}_{#1}}         
\newcommand{\pDa}[1]{\polyset{D}_{#1}}        
\newcommand{\pD}[1]{\polyset{D}_{#1}}     

\newcommand{\pT}[1]{\polyset{T}_{#1}}     
\newcommand{\pE}[1]{\polyset{E}_{#1}}     

\DeclareMathOperator{\len}{len}
\DeclareMathOperator{\mult}{mult}
\newcommand{\refine}{\mathbin{/\!\!/}}

\newcommand{\mat}{\mathcal{M}}
\newcommand{\nat}{\mathcal{N}}

\begin{document}

\title{Tame Decompositions and Collisions}
\pdftitle{Tame Decompositions and Collisions}
\author{
Konstantin Ziegler\\
B-IT, Universit\"at Bonn\\
D-53113 Bonn, Germany\\
\email{zieglerk@bit.uni-bonn.de}\\
\url{http://cosec.bit.uni-bonn.de/}
}
\pdfauthor{Konstantin Ziegler}

\maketitle


\begin{abstract}
  A univariate polynomial $f$ over a field is decomposable if $f= g
  \circ h= g(h)$ for nonlinear polynomials $g$ and $h$.  It is
  intuitively clear that the decomposable polynomials form a small
  minority among all polynomials over a finite field.  The tame case,
  where the characteristic $p$ of $\Fq$ does not divide $n = \deg f$,
  is fairly well-understood, and we have reasonable bounds on the
  number of decomposables of degree $n$.  Nevertheless, no exact
  formula is known if $n$ has more than two prime factors.  In order to count
  the decomposables, one wants to know, under a suitable
  normalization, the number of collisions, where essentially different
  $(g, h)$ yield the same $f$.  In the tame case, Ritt's Second Theorem
  classifies all 2-collisions.

  We introduce a normal form for multi-collisions of decompositions of
  arbitrary length with exact description of the (non)uniqueness of
  the parameters.  We obtain an efficiently computable formula for the
  exact number of such collisions at degree $n$ over a finite field of
  characteristic coprime to $p$.  This leads to an algorithm for the exact number of decomposable polynomials at degree $n$
  over a finite field $\Fq$ in the tame case.
\end{abstract}






\begin{sagesilent}
load('../../code/decode.sage')
\end{sagesilent}

\section{Introduction}
\label{sec:introduction}

The \emph{composition} of two univariate polynomials $g,h \in F[x]$ over a field
$F$ is denoted as $f= g \circ h= g(h)$, and then $(g,h)$ is a
\emph{decomposition} of $f$, and $f$ is \emph{decomposable} if $g$ and
$h$ have degree at least $2$.  In the 1920s, Ritt, Fatou, and Julia
studied structural properties of these decompositions over
$\mathbb{C}$, using analytic methods. Particularly important are two
theorems by Ritt on the uniqueness, in a suitable sense, of
decompositions, the first one for (many) indecomposable components and
the second one for two components, as above.  \cite{eng41} and
\cite{lev42} proved them over arbitrary fields of characteristic zero
using algebraic methods.

The theory was extended to arbitrary characteristic by \cite{frimac69}, \cite{dorwha74}, \cite{sch82c,
  sch00c}, \cite{zan93}, and others. Its use in a cryptographic
context was suggested by \cite{cad85}. In computer algebra, the
decomposition method
of \cite{barzip85} requires exponential time.  A fundamental
dichotomy is between the \emph{tame case}, where the characteristic
$p$ does not divide $\deg g$, and the \emph{wild case}, where $p$
divides $\deg g$, see \cite{gat90d,gat90c}.
A breakthrough result of \cite{kozlan89} was their
polynomial-time algorithm to compute tame decompositions; see also
\cite*{gatkoz87}; \cite*{kozlan96}; \cite{gutsev06},
and the survey articles of \cite{gat02c} and
\cite{gutkoz03} with further references.

Schur's
conjecture, as proven by \cite{tur95}, offers a natural connection
between indecomposable polynomials with degree coprime to $p$ and
certain absolutely irreducible bivariate polynomials.
On a different, but related topic, \cite{avazan03} study ambiguities in the
decomposition of rational functions over $\mathbb{C}$.

It is intuitively clear that the univariate decomposable polynomials
form only a small minority among all univariate polynomials over a
field.  A set of distinct decompositions of $f$ is called a
\emph{collision}. The number of decomposable polynomials of degree $n$
is thus the number of all pairs $(g,h)$ with $\deg g \cdot \deg h = n$
reduced by the ambiguities introduced by collisions.  An important
tool for estimating the number of collisions is Ritt's Second Theorem.
Ritt worked with $F=\mathbb{C}$ and used analytic
methods. Subsequently, his approach was replaced by algebraic methods
and Ritt's Second Theorem was also shown to hold in positive
characteristic $p$. The original versions of this required
$p>\deg(g\circ h)$. \cite{zan93} reduced this to the milder and more
natural requirement $g' \neq 0$ for all $g$ in the collision.  His
proof works over an algebraic closed field, and Schinzel's
(\citeyear{sch00c}) monograph adapts it to finite fields.

The task of counting compositions over a finite field of
characteristic $p$ was first considered in
\cite{gie88b}.
\Citet{gat08c} presents general approximations to the
number of decomposable polynomials. These come with satisfactory
(rapidly decreasing) relative error bounds except when $p$ divides $n
= \deg f$ exactly twice.  \cite{blagat13} determine exactly the
number of decomposable polynomials in one of these difficult cases,
namely when $n = p^{2}$.

\cite{zan08} studies a different but related question, namely
compositions $f= g \circ h$ in $ \mathbb{C} [x]$ with a \emph{sparse}
polynomial $f$, having $t$ terms. The degree is not bounded.  He gives
bounds, depending only on $t$, on the degree of $g$ and the number of
terms in $h$.  Furthermore, he gives a parametrization of all such
$f$, $g$, $h$ in terms of varieties (for the coefficients) and
lattices (for the exponents).  \citet{boddeb09} also deal with
counting.

\cite{ziemue08} derive an efficient method for describing all complete
decompositions of a polynomial, where all components are
indecomposable.  This turns Ritt's First Theorem into an applicable
form and \cite{medsca14} combine this approach with results from model
theory to describe the subvarieties of the $k$-dimensional affine space
that are preserved by a coordinatewise polynomial map.  Both works
lead to slightly different canonical forms for the complete
decomposition of a given polynomial.  \cite{ziemue08} employ
\emph{Ritt moves}, where adjacent indecomposable $g,h$ in a complete
decomposition are replaced by $g^{*},h^{*}$ with the same composition,
but $\deg g = \deg h^{*} \neq \deg h = \deg g^{*}$.  Such collisions
are the theme of Ritt's Second Theorem and \cite{gat12a} presents a
normal form  with an exact description of the
(non)uniqueness of the parameters.

Our work combines the ``normalizations'' of Ritt's theorems by
\cite{ziemue08} and \cite{gat12a} to classify collisions of two or
more decompositions, not necessarily complete and of arbitrary length.
We make the following contributions.
\begin{itemize}
\item We obtain a normal form for collisions described by a
  set of degree sequences for (possibly incomplete)
  decompositions. (\autoref{thm:split} and \autoref{thm:class})
\item The (non)uniqueness of the parameters leads to an exact formula
  for the number of such collisions over a finite field with characteristic coprime
  their degree. (\autoref{thm:count})
\item We conclude with an efficient algorithm for the number
  of decomposable polynomials at degree $n$ over a finite field of
  characteristic coprime $n$. (\autoref{algo:count})
\end{itemize}
The latter extends the explicit formulae of \cite{gat08c} for $n$ a
semiprime or the cube of a prime.

We proceed in three steps.  In \autoref{sec:notation}, we introduce
notation and establish basic relations.  In \autoref{sec:structure},
we introduce the \emph{relation graph} of a set of collisions which
captures the necessary order and possible Ritt moves for any in
decomposition.  This leads to a complete classification of collisions
by \autoref{thm:split} and \autoref{thm:class}.  We conclude
with the corresponding formula for the number of such collisions over
a finite field and the corresponding procedure in \autoref{sec:counting}.

\section{Notation and Preliminaries}
\label{sec:notation}

A nonzero
polynomial $f\in F[x]$ over a field $F$ of characteristic $p \geq 0$ is \emph{monic} if
its leading coefficient $\operatorname{lc}(f)$ equals $1$.  We call $f$
\emph{original} if its graph contains the origin, that is, $f(0)=0$.
  For $g, h \in F[x]$,
\begin{equation}
  \label{defComp:f}
  f = g \circ h = g(h) \in F[x]
\end{equation}
  is their \emph{composition}.  If $\deg g, \deg h \geq 2$, then
  $(g,h)$ is a \emph{decomposition} of $f$. A polynomial $f \in F[x]$
  is \emph{decomposable} if there exist such $g$ and $h$, otherwise
  $f$ is \emph{indecomposable}.  A decomposition \eqref{defComp:f} is \emph{tame} if $p\nmid \deg g$, and
$f$ is \emph{tame} if $p\nmid \deg f$.

Multiplication by a unit or addition of a constant does not change
decomposability, since
\begin{equation}
  \label{eq:82}
  f = g \circ h \Longleftrightarrow a f+b = (a g+b) \circ h
\end{equation}
for all $f$, $g$, $h$ as above and $a,b \in F$ with $a\neq 0$.  In
other words, the set of decomposable polynomials is invariant under
this action of $F^{\times} \times F$ on $F[x]$.  Furthermore, any
decomposition $(g,h)$ can be normalized by this action, by taking $a =
\operatorname{lc} (h)^{-1} \in F^{\times}$, $b=-a \cdot h(0) \in F$,
$g^{*} = g((x-b)a^{-1}) \in F[x]$, and $h^{*} = ah+b$.  Then $g\circ h
= g^{*} \circ h^{*}$ and $g^{*}$ and $ h^{*}$ are monic original.

It is therefore sufficent to consider compositions $f = g \circ h$
where all three polynomials are monic original.  For $n \geq
1$ and any positive divisor $d$ of $n$, we write
\begin{align}
  \pP{n}(F) & = \{f \in F[x] \colon \text{$f$ is monic original of
    degree $n$} \}, \\
\pDa{n}(F) & = \{f \in \pP{n} \colon \text{$f$ is decomposable} \}, \\
\pD{n,d}(F) & = \{f \in \pP{n} \colon f = g \circ h \text{
  for some $(g, h) \in \pP{d} \times \pP{n/d}$} \}.
\end{align}
We sometimes leave out $F$ from the notation
when it is clear from the context and have over a finite field
$\Fq$ with $q$ elements,
\begin{equation}
  \label{eq:42}
  \# \pP{n} = q^{n-1}. \tag{2.1a}
\end{equation}
It is well known that in a tame decomposition, $g$ and $h$ are
uniquely determined and we have over $\Fq$
\begin{equation}
  \label{eq:20}
  \# \pD{n, d} = q^{n+n/d-2}
\end{equation}
if $n$ is coprime to $p$.

The set $\pDa{n}$ of all decomposable polynomials in $\pP{n}$
satisfies
\begin{equation}
  \label{eq:8}
  \pDa{n}= \bigcup_{\substack{d\mid n\\1<d<n}} \pD{n,d}.
\end{equation}
In particular, $\pDa{n} = \varnothing$ if $n$ is prime. Our
collisions turn up in the resulting inclusion-exclusion formula for
$\# \pDa{n}$ if $n$ is composite.

Let $N = \{ 1< d <n \colon d \mid n \}$ be the set of nontrivial
divisors of $n$ and $D \subseteq N$ a
nonempty subset of size $k$.  This defines a set
\begin{equation}
  \label{eq:22}
  \pD{n, D} = \bigcap_{d \in D} \pD{n, d}
\end{equation}
of \emph{$k$-collisions}.  We obtain from \eqref{eq:8} the inclusion-exclusion formula
\begin{equation}
\label{eq:9}
  \# \pDa{n} = \sum_{k \geq 1} (-1)^{k+1} \sum_{\substack{D \subseteq \divisors \\ \# D = k}} \# \pD{n, D}.
\end{equation}
For $\# D = 1$, the size of $\pD{n, D}$ is given in \eqref{eq:20}.
For $\# D = 2$, the central tool for understanding is Ritt's Second
Theorem as presented in the next subsection.

For $f \in P_{n} (F)$ and $a \in F$, the \emph{original shift} of $f$ by $a$ is
\begin{equation}
  \label{eq:84}
  f^{[a]} = (x-f(a)) \circ f \circ (x+a) \in \pP{n}(F).
\end{equation}
Original shifting defines a group action of the additive group of $F$ on $\pP{n}(F)$.  Shifting respects decompositions in the sense that
for each decomposition $(g,h)$ of $f$ we have a decomposition $(g^{[h(a)]}, h^{[a]})$ of $f^{[a]}$, and vice versa.
We denote $(g^{[h(a)]}, h^{[a]})$ as $(g, h)^{[a]}$.  The stabilizer
of a monic original polynomial $f$ under original shifting is $F$ if
$f$ is linear and $\{0\}$ otherwise.
\noeqref{eq:84}

\subsection{Normal Form for Ritt's Second Theorem}
\label{sec:ritt-2}

In the 1920s, Ritt, Fatou, and Julia investigated the composition
$f=g\circ h=g(h)$ of univariate polynomials over a field $F$ for $F =
\mathbb{C}$. It emerged as an important question to determine the
collisions (or nonuniqueness) of such decompositions, that is,
different components $(g,h)\neq(g^{*},h^{*})$ with equal composition
$g\circ h=g^{*}\circ h^{*}$ and equal sets of degrees: $\deg g = \deg
h^{*} \neq \deg h = \deg g^{*}$.

\cite{rit22} presented two types of essential collisions:
\begin{align}
x^{e}\circ x^{k}w(x^{e}) & =
x^{ke}w^{e}(x^{e})=x^{k}w^{e}\circ
x^{e}, \label{eq:76} \\
T_{d}^{*}(x,z^{e})\circ T^{*}_{e}(x,z) & = T^{*}_{e d}(x,z)=T^{*}_{e}(x,z^{d})\circ
T_{d}^{*}(x,z), \label{eq:28}
\end{align}
where $w\in F[x]$, $z\in F^{\times} = F \setminus \{0\}$, and $T_{d}^{*}$
is the $d$th \emph{Dickson polynomial of the first kind}. And then he proved
that these are all possibilities up to composition with linear polynomials.
This involved four unspecified linear functions, and it is not clear whether there is a relation
between the first and the second type of example.  Without loss of
generality, we use the \emph{originalized $d$th Dickson polynomial}
$T_{d}(x,z) = T_{d}^{*}(x,z) - T_{d}^{*}(0,z)$ which also satisfy \eqref{eq:28}.

\Citet{gat12a} presents a normal form for the decompositions in Ritt's
Theorem under Zannier's assumption $g'(g^{*})'\neq 0$ and the standard
assumption $\gcd(e,d)=1$, where $d=k+e\deg w$ in
\eqref{eq:76}. This normal form is unique unless $p\mid m$.

\begin{theorem} (Ritt's Second Theorem, Normal Form, tame case)
  \label{thm:Ritt2}
  Let $d > e \geq 2$ be
  coprime integers, and $n=de$ coprime to the characteristic of $F$.  Furthermore, let $f = g \circ h =
  g^{*} \circ h^{*}$ be monic original polynomials with $\deg g = \deg h^{*} = d$, $\deg h = \deg g^{*} = e$.

Then either \ref{th:fifi-1} or \ref{th:fifi-2} holds, and \ref{th:fifi-3} is also
  valid.
  \begin{ronumerate}
  \item\label{th:fifi-1} (Exponential Case) There exists a monic
    polynomial $w \in F[x]$ of degree $s$ and $a \in F$ so that
    \begin{equation}\label{eq:mopo}
      f= (x^{ke}w^{e}(x^{e}))^{[a]}
    \end{equation}
   where $d=se+k$ is the division with remainder of $d$ by $e$, with
   $1 \leq k < e$. Furthermore
    \begin{align}
      \label{eq:unidet-1} (g,h)& = (x^{k}w^{e},x^{e})^{[a]}, & (g^{*}, h ^{*}) & = (x^{e}, x^{k} w
      (x^{e}))^{[a]},
    \end{align}
    and $(w,a)$ is uniquely determined by $f$ and $d$.
    Conversely, any $(w,a)$ as above yields a $2$-collision via the above
    formulas.
  \item\label{th:fifi-2} (Trigonometric Case) There exist $z,a \in F$
    with $z \neq 0$ so that
    \begin{align}\label{eq:TN}
      f = T_{n}(x,z)^{[a]}.
    \end{align}
    Furthermore we have
    \begin{align}
      \label{eq:ab} (g,h) & =   (T_{d}(x,z^{e}) , T_{e}(x,z))^{[a]}, &
      (g^{*}, h^{*})& = (T_{e}(x,z^{d}),
      T_{d}(x,z))^{[a]},
    \end{align}
    and $(z,a)$ is uniquely determined by $f$.     Conversely, any $(z,a)$ as
    above yields a $2$-collision via the above formulas.

  \item\label{th:fifi-3} For $e=2$, the Trigonometric Case is included in
    the Exponential Case.  For $e \geq 3$, the Exponential and Trigonometric Cases
    are mutually exclusive.
  \end{ronumerate}
\end{theorem}

If $p \nmid n$, then the case where $\gcd (d, e) \neq 1$ is reduced
to the previous one by the following result about the left and right
greatest common divisors of decompositions.  It was shown over
algebraically closed fields by \citet[Proposition~1]{tor88a}; a more concise proof
using Galois theory is due to \citet[Lemma~2.8]{ziemue08}.  We use the
version of \citet[Fact~6.1(i)]{gat12a}, adapted to monic original polynomials.
\begin{proposition}
  \label{pro:tortrat}
  Let $d,e,d^{*},e^{*} \geq 2$ be integers and $de=d^{*}e^{*}$ coprime to $p$.
  Furthermore, let $g \circ h = g^{*} \circ h^{*}$ be monic
  original polynomials with $\deg g = d$, $\deg h = e$, $\deg
  g^{*} = d^{*}$, $\deg h^{*} = e^{*}$, and $\ell =
  \gcd(d,d^{*})$, $r = \gcd(e,e^{*})$.  Then there are unique
  monic original polynomials $a$ and $b$ of degree $\ell$ and $r$,
  respectively, such that
  \begin{align}
   g & = a \circ u, & h & = v \circ b, \\
   g^{*} & = a \circ u^{*}, & h^{*} & = v^{*} \circ b,
  \end{align}
  for unique monic original polynomials $u$, $u^{*}$, $v$, $v^{*}$ of degree
  $d/\ell$, $d^{*}/\ell$, $e/r$, and $e^{*}/r$, respectively.
\end{proposition}
This determines $\pD{n,\{d,e\}}$ exactly if $p
\nmid n = de$.

For coprime integers $d \geq 2$ and $e \geq 1$, we define the sets
\begin{align}
  \pE{d,e} & = \begin{cases*}
\pP{d} & for $e = 1$, \\
\{x^{k}w^{e} \in \pP{d} \colon d = s \cdot e + k \text{ with } 1 \leq k
< e & \label{eq:25} \\
\quad \text{ and } w \in \Fq[x] \text{ monic of degree } s\},  & otherwise,
\end{cases*} \\
  \pT{d,e} & = \{ T_{d}(x,z^{e}) \in \pP{d} \colon z \in
  \Fq^{\times}\}
\end{align}
of \emph{exponential} and \emph{trigonometric components},
respectively.  For $d < e$, we have $s = 0$, $k = d$ in \eqref{eq:25}, and therefore
\begin{equation}
  \label{eq:15}
  \pE{d,e} = \{ x^{d} \}. \tag{2.10a}
\end{equation}

This allows the following reformulation of \autoref{thm:Ritt2}.
\begin{corollary}
  \label{cor:Ritt2}
  Let $f \in \pD{n, \{d, e\}}$.  Then either (i) or (ii) holds and
  (iii) is also valid.
  \begin{ronumerate}
    \item There is a unique monic original $g \in \pE{d, e}$ and a
      unique $a \in F$ such that
     \begin{equation}
        \label{eq:7}
        f = (g \circ x^{e})^{[a]}.
      \end{equation}
    \item There is a unique monic original $g \in \pT{de, 1}$ and a
      unique $a \in F$ such that
      \begin{equation}
        \label{eq:27}
        f = g^{[a]}.
      \end{equation}
    \item If $e = 2$, then case (ii) is included in case (i).  If $e
      \geq 3$, they are mutually exclusive.
  \end{ronumerate}
   Conversely, we have
  \begin{equation}
    \label{eq:33}
    \pD{n, \{d, e\}} = (\pE{d,e} \circ \pE{e, d})^{[F]} \cup \pT{de,
    1}^{[F]},
  \end{equation}
where the union is disjoint if and only if $e \geq 3$, and
\begin{equation}
  \# \pD{n, \{d, e\}} = q \cdot (q^{\floor{d/e}} + (1 - \delta_{e,2})(q-1)).
\end{equation}
\end{corollary}

With respect to the size under original shifting, we have the following consequences.

\addtocounter{equation}{2}

\begin{proposition}
  \label{pro:6}
  For $F = \Fq$, coprime $d \geq 2$ and $e \geq 1$, both coprime to
  $p$, we have
  \begin{align}
    \# \pT{d,e}^{[\Fq]} & = \begin{cases*}
      q & for $d = 2$, \\
      q (q-1)/\gcd(q-1, e) & otherwise,
    \end{cases*} \\
    \# \pE{d,e}^{[\Fq]} & = \begin{cases*}
      q^{d-1} & for $e = 1$, \\
      q^{\floor{d/2}+1} - q(q-1)/2 & for $e=2$, \\
      q^{\floor{d/e}+1} & otherwise.
    \end{cases*}
  \end{align}
\end{proposition}

\begin{proof}
  \begin{ronumerate}
    \item For $d=2$, we have $\pT{2,e}=\{x^{2}\}$ independent from $e$.  Since
  $p \neq 2$, the original shifts $(x^{2})^{[a]} = x^{2} + 2ax$ range
  over all monic original polynomials of degree 2 as $a$ runs over all
  field elements.  Thus $\pT{2,e}^{[F]} = \pP{2}$ and the size follows
  from \eqref{eq:42}.

  For $d > 2$, the coefficient of $x^{d-2}$ in $T_{d}(x, z^{e}) \in
  \pT{d,e}$ is $-dz^{e}$. Since there are exactly
  $(q-1)/\gcd(q-1, e)$ distinct $e$th powers $z^{e}$ for nonzero
  elements $z \in \Fq$, this shows
  \begin{equation}
    \label{eq:101}
    \# \pT{d,e} = (q-1)/\gcd(q-1, e).
  \end{equation}
  For the claimed formula it is sufficient to show that for $T \in
  \pT{d,e}$ and $a \in F$, we have
  \begin{equation}
    \label{eq:102}
    T^{[a]} \in \pT{d,e} \text{ if and only if } a=0.
  \end{equation}
  For $d$ odd, we also have $T$ an odd polynomial.  Hence the
  coefficient of $x^{d-1}$ in $T$ is $0$ and the coefficient of
  $x^{d-1}$ in $T^{[a]}$ is $ad$. This proves the claim. For $d$ even,
  the same argument applies with ``odd'' replaced by ``even''.
\end{ronumerate}

For a nonzero polynomial $f \in F[x]$ and $b$ in some
algebraic closure $K$ of $F$, let $\mult_{b}(f)$ denote the \emph{root
multiplicity} of $b$ in $f$, so that $f = (x-b)^{\mult_{b}(f)}u$
with $u \in K[x]$ and $u(b) \neq 0$.

\begin{ronumerate}[resume]
  \item For $e = 1$, we have $\pE{d,e} = \pP{d} = \pP{d}^{[F]}$ of size
    $q^{d-1}$ by \eqref{eq:42}.

    For $e > 1$, we have $\# \pE{d,e} = q^{\floor{d/e}}$. It is sufficient
    to show that for $f \in \pE{d,e}$ and $a \in F$, we have
    \begin{equation}
      \label{eq:67}
      f^{[a]} \in \pE{d,e} \text{ if and only if } a = 0.
    \end{equation}
    We have directly $f^{[0]} = f \in \pE{d,e}$. Conversely, let
    $f^{[a]}= \bar{f} = x^{k} \bar{w}^{e} \in \pE{d,e}$, and compare the
    derivatives
    \begin{align}
      {f^{[a]}}' & = (x+a)^{k-1} w(x+a)^{e-1} (kw(x+a) +
      e(x+a)w'(x+a)), \label{eq:87} \\
      {\bar{f}}' & = x^{k-1}\bar{w}^{e-1} (k\bar{w} + ex{\bar{w}}'), \label{eq:88}
    \end{align}
    respectively. These are nonzero, since $p \nmid d = \deg(\bar{f}) =
    \deg(f^{[a]})$, and we compute the root multiplicity of $0$ as
    \begin{align}
      \mult_{0} ({f^{[a]}}') & = [a=0]\cdot(k-1) + \mult_{a}(w) \cdot (e-1) + \begin{cases*}
        \mult_{a}(w) & if $p \mid \mult_{a}(w)$, \\
        \mult_{a}(w) - [a \neq 0] & otherwise.
      \end{cases*} \\
                           & = e \mult_{a}(w) + [a=0] (k-1) - [p \nmid
                           \mult_{a}(w) \text{ and } a \neq 0], \\
      \mult_{0}({\bar{f}}') & = k-1 + (e-1) \cdot \mult_{0}(\bar{w}) +
      \mult_{0}(\bar{w}) \\
                           & = e \mult_{0}(\bar{w}) + k -1.
    \end{align}
    If $f^{[a]} = \bar{f}$, we find modulo $e$
    \begin{equation}
      \label{eq:70}
      k -1 = [a=0]\cdot(k-1) - [p \nmid \mult_{a}(w) \text{ and } a \neq 0].
    \end{equation}
    This holds if
    \begin{itemize}
      \item $a=0$ or
      \item $p \mid \mult_{a}(w)$ and $k=1$.
    \end{itemize}
    It remains to show that the latter case is included in the
    former. In other words, that $p \mid \mult_{a}(w)$ and $k = 1$
    imply $a = 0$. Let $\Delta = w(x+a) - \bar{w}$ of degree less than
    $s$, since both are monic. Then
    \begin{align}
      0 = {f^{[a]}}' - {\bar{f}}' & = w(x+a)^{e} - {\bar{w}}^{e} +
      e((x+a)w'(x+a) - x{\bar{w}}') \\
      & = e \Delta {\bar{w}}^{e-1} +
      \binom{e}{2}\Delta^{2}{\bar{w}}^{e-2} + \dots \\
      & \quad + (e-1)x\bar{w}' + ex\Delta' + ea\bar{w}' + ea\Delta'
    \end{align}
    If $\Delta = 0$, we are done. Otherwise, $\deg \Delta \geq 0$ and
    the coefficient of $x^{\deg \Delta + s(e-1)}$ is
    \begin{equation}
      \label{eq:103}
      e \operatorname{lc}(\Delta) + s(e-1) [e=2 \text{ and } \deg
      \Delta = 0].
    \end{equation}
    If $e > 2$, this is nonzero, a contradiction.

    We have $w(x+a) = \bar{w}$.  Feeding this and $k=1$ back into the definition of
    $f^{[a]}$ and $\bar{f}$, we obtain for their difference
    \begin{equation}
      \label{eq:86}
      f^{[a]} - \bar{f} = (x+a){\bar{w}}^{e} - f(a) - x{\bar{w}}^{e} =
      a{\bar{w}}^{e} - f(a).
    \end{equation}
    With $\deg(w) = s > 0$ this implies $a=0$.

  \end{ronumerate}
\end{proof}

\section{Normal Form for Collisions}
\label{sec:structure}

The results of the previous section suffice to describe $2$-collisions
of decompositions $g \circ h = g^{*} \circ h^{*}$ with length $2$
each.  This section describes the structure of ``many''-collisions of
decompositions with arbitrary, possibly pairwise distinct, lengths.
Let $\vec{d} = (d_{1}, d_{2}, \dots, d_{\ell})$ be an ordered
factorization of $n = d_{1} \cdot d_{2} \cdot \dots \cdot d_{\ell}$ with
$\ell$ nontrivial divisors $d_{i} \in N$, $1 \leq i \leq \ell$, and
define the set
\begin{equation}
  \label{eq:32}
  \pD{n, \vec{d}} = \{f \in \pP{n} \colon f = g_{1} \circ
  \dots \circ g_{\ell} \text{ with } \deg g_{i} = d_{i} \text{ for all }
  1 \leq i \leq \ell\}
\end{equation}
of decomposable polynomials with decompositions of \emph{length}
$\ell$ and \emph{degree sequence} $\vec{d}$.  For a set $\vec{D} = \{\vec{d}^{(1)}, \vec{d}^{(2)}, \dots, \vec{d}^{(c)} \}$ of $c$
 ordered factorizations of $n$, we define
\begin{align}
  \label{eq:34}
  \pD{n, \vec{D}} & = \bigcap_{\vec{d} \in \vec{D}} \pD{n, \vec{d}} \\
  & =  \{  f \in \pP{n} \colon f = g^{(k)}_{1} \circ
  \dots \circ g^{(k)}_{\ell_{k}} \text{ with } \deg g_{i}^{(k)} = d_{i}^{(k)} \\
& \quad \quad \text{ for
  all } 1 \leq k \leq c \text{ and } 1 \leq i \leq \ell_{k}\}.
\end{align}
For $\# \vec{D} = 1$, we have
  \begin{align}
\label{eq:5}
\pD{n, \vec{D}} = \pD{n, \vec{d}} & = \pP{d_{1}} \circ \pP{d_{2}} \circ \dots \circ
\pP{d_{\ell}}, \\
\# \pD{n, \vec{D}} = \# \pD{n, \vec{d}} & = q^{\sum_{1 \leq i \leq \ell}
  d_{i} - \ell},
  \end{align}
where $\vec{D} = \{ \vec{d}\}$ and $\vec{d} = (d_{1}, d_{2}, \dots,
d_{\ell})$.  The rest of this section deals with $\# \vec{D} > 1$.

We determine the structure of $\pD{n, \vec{D}}$.  First, we replace
$\vec{D}$ by a \emph{refinement} $\vec{D}^{*}$, where all
elements are suitable permutations of the \emph{same} ordered factorization of $n$.  Second,
we define the \emph{relation graph} of $\vec{D}^{*}$ that captures the
degree sequences for polynomials in
$\pD{n, \vec{D}}$.  Finally, we classify the elements of
$\pD{n, \vec{D}}$ as a composition of unique trigonometric or unique
exponential components as defined in \eqref{eq:25}.

\subsection{A refinement of \large $\vec{D}$}
\label{sec:transf-into-copr}

Let $\vec{d} = (d_{1}, d_{2})$ and $\vec{e} = (e_{1}, e_{2})$ be
distinct ordered factorizations of $n$.  Let $\ell = \gcd(d_{1},
e_{1})$, $d_{1}^{*} = d_{1}/\ell$, $e_{1}^{*} = e_{1}/\ell$, $r =
\gcd(d_{2}, e_{2})$, $d_{2}^{*} = d_{2}/r$, and $e_{2}^{*} = e_{2}/r$.  Then
\autoref{pro:tortrat} shows
\begin{equation}
  \label{eq:37}
  \pD{n, \{\vec{d}, \vec{e}\}} = \pD{n, \{\vec{d}^{*}, \vec{e}^{*}\}}
\end{equation}
for $\vec{d}^{*} = (\ell, d_{1}^{*}, d_{2}^{*}, r)$ and $\vec{e}^{*}
= (\ell, e_{1}^{*}, e_{2}^{*}, r)$ with $\gcd(d_{1}^{*},
e_{1}^{*}) = 1 = \gcd(d_{2}^{*}, e_{2}^{*})$ and therefore $d_{1}^{*} = e_{2}^{*}$
and $d_{2}^{*} = e_{1}^{*}$.  We generalize this procedure to
two ordered factorizations of arbitrary length.  For squarefree $n$ this is
similar to the computation of a coprime (also: $\gcd$-free) basis for $\{d_{1}, d_{2}, e_{1}, e_{2}\}$,
if we keep duplicates and the order of factors; see
\citet[Section~4.8]{bacsha97}.  For squareful $n$, the factors with
$\gcd > 1$ require additional attention.

Let $\vec{d} = (d_{1}, d_{2}, \dots, d_{\ell})$ be an ordered factorization
of $n$ and call the underlying unordered multiset $\uset{\vec{d}} =
\{d_{1}, d_{2}, \dots, d_{\ell}\}$ of divisors its \emph{basis}.  A
\emph{refinement} of $\vec{d}$ is an ordered factorization
$\vec{d^{*}} = (d^{*}_{11}, \dots, d^{*}_{1m_{1}}, d^{*}_{21},
\dots, d^{*}_{2m_{2}}, \dots, d^{*}_{\ell 1}, \dots,
d^{*}_{\ell m_{\ell}})$, where $d_{i} = \prod_{1 \leq k \leq m_{i}}
d^{*}_{ik}$ for all $1 \leq i \leq \ell$.  We write $\vec{d}^{*} \mid
\vec{d}$ and have directly
\begin{equation}
  \label{eq:24}
  \pD{n, \vec{d}^{*}} \subseteq \pD{n, \vec{d}}.
\end{equation}
Every ordered factorization is a refinement of $(n)$.  A
\emph{complete} refinement of $\vec{d} = (d_{i})_{1 \leq i \leq \ell}$ is obtained by replacing every
$d_{i}$ by one of its ordered factorization into primes.

Two ordered factorizations $\vec{d} = (d_{1}, \dots, d_{\ell})$ and $\vec{e} =
(e_{1}, \dots, e_{\ell})$ of $n$ with the same basis, define a permutation
$\sigma = \sigma (\vec{d}, \vec{e})$ on the indices $1,2,\dots,\ell$
through
\begin{equation}
  \label{eq:38}
d_{i} = e_{\sigma(i)}
\end{equation}
for $1 \leq i \leq \ell$.  We require
\begin{equation}
  \label{eq:83}
  \sigma(i) < \sigma(j) \text{ for all } i < j \text{ with } d_{i} = d_{j}
\end{equation}
to make $\sigma$ unique.
In other words, $\sigma$ has to preserve the order of repeated
divisors.  If even stronger,
\begin{equation}
  \label{eq:13}
  \sigma(i) < \sigma(j) \text{ for all } i < j \text{ with }
  \gcd(d_{i},d_{j}) > 1,
\end{equation}
then we call $\vec{d}$ and $\vec{e}$ \emph{associated}.  Any two complete
refinements $\vec{d}^{*}$ of $\vec{d}$ and $\vec{e}^{*}$ of $\vec{e}$,
respectively, are associated and we have from \eqref{eq:24}
\begin{equation}
  \label{eq:14}
\pD{n, \{\vec{d}^{*},\vec{e}^{*}\}} \subseteq \pD{n, \{\vec{d},\vec{e}\}}.
\end{equation}
We ask for associated refinements $\vec{d}^{*}$ and $\vec{e}^{*}$, respectively,
that describe the same set of collisions as $\vec{d}$ and $\vec{e}$.  \autoref{algo:refine}
solves this task and returns ``coarsest'' associated refinements $\vec{d}^{*}$
and $\vec{e}^{*}$ that yield equality in \eqref{eq:14}.  We call the
output $\vec{d}^{*}$ of \autoref{algo:refine}
the \emph{refinement of $\vec{d}$ by $\vec{e}$} and denote it by
$\vec{d}^{*} = \vec{d} \refine \vec{e}$.  Similarly, $\vec{e}^{*} =
\vec{e} \refine \vec{d}$ is the refinement of $\vec{e}$ by $\vec{d}$
and this is well-defined, since interchanging the order of the input
merely interchanges the order of the output.

\begin{algorithm2e}
\caption{Refine $\vec{d}$ and $\vec{e}$}
\label{algo:refine}

\KwIn{two ordered factorizations $\vec{d} = (d_{1}, \dots, d_{\ell})$ and
  $\vec{e} = (e_{1}, \dots, e_{m})$ of $n$}
\KwOut{two associated refinements $\vec{d}^{*} \mid \vec{d}$ and
  $\vec{e}^{*} \mid \vec{e}$}

$\vec{d}^{*} \gets \begin{pmatrix}
1 & \dots & 1 & d_{1} \\
1 & \dots & 1 & d_{2} \\
  & \vdots & & \\
1 & \dots & 1 & d_{\ell} \\
\end{pmatrix} = (d^{*}_{i,j})_{\substack{1 \leq i \leq \ell \\ 1 \leq j
    \leq m+1}}$\;

$\vec{e}^{*} \gets \begin{pmatrix}
1 & \dots & 1 & e_{1} \\
1 & \dots & 1 & e_{2} \\
  & \vdots & & \\
1 & \dots & 1 & e_{m} \\
\end{pmatrix} = (e^{*}_{j,i})_{\substack{1 \leq j \leq m \\ 1 \leq m \leq \ell+1}}$\;

\For{$i = 1, \dots, \ell$}{
\For{$j = 1, \dots, m$ \label{step:inn_in}}{
$c \gets \gcd(d^{*}_{i,m+1}, e^{*}_{j,\ell+1})$ \label{step:gcd}\;
$d_{i,j}^{*} \gets c$ and $e_{j,i}^{*} \gets c$ \label{step:split_left}\;
$d^{*}_{i,m+1} \gets d^{*}_{i,m+1}/c$ and $e^{*}_{j,\ell+1} \gets
e^{*}_{j,\ell+1}/c$ \label{step:split_right}\;
}\label{step:inn_out}
}
remove last column (of all $1$'s) from $\vec{d}^{*}$ and $\vec{e}^{*}$\;
$\vec{d}^{*} \gets (d_{i,j}^{*})_{\substack{1 \leq k \leq \ell m \\ k
    = (i-1)m + j \\ 1 \leq i \leq \ell, 1 \leq j \leq m}}$
\tcc{rewrite row-by-row as a sequence}
$\vec{e}^{*} \gets (e_{j,i}^{*})_{\substack{1 \leq k \leq m \ell \\ k
    = (j-1)\ell + i \\ 1 \leq j \leq m, 1 \leq i \leq \ell}}$
\tcc{rewrite row-by-row as a sequence}
remove all $1$'s from $\vec{d}^{*}$ and $\vec{e}^{*}$\;
\KwRet{$\vec{d}^{*}$, $\vec{e}^{*}$}
\end{algorithm2e}

\begin{lemma}
  \label{lem:9}
  For two ordered factorizations $\vec{d}$ and $\vec{e}$ of $n$, the following are equivalent.
  \begin{ronumerate}
    \item\label{it:9:1} $\len(\vec{d} \refine \vec{e}) = \len(\vec{d})$ and $\len(\vec{e} \refine \vec{d}) = \len(\vec{e})$.
    \item\label{it:9:2} $\vec{d} \refine \vec{e} = \vec{d}$ and $\vec{e} \refine \vec{d} = \vec{e}$.
    \item\label{it:9:3} $\vec{d}$ and $\vec{e}$ are associated.
  \end{ronumerate}
\end{lemma}

\begin{proof}
  Let $\vec{d} = (d_{1}, \dots, d_{\ell})$ and $\vec{e} = (e_{1},
  \dots, e_{\ell})$ be associated and $\sigma =
  \sigma(\vec{d}, \vec{e})$ the unique permutation satisfying
  \eqref{eq:38} and \eqref{eq:13}. Then we have in \autoref{step:gcd} of \autoref{algo:refine}
  \begin{equation}
    \label{eq:73}
    \gcd(d_{i, m+1}^{*}, e_{j,\ell+1}^{*}) = \begin{cases*}
      d_{i} = e_{j} & if $j = \sigma(i)$, \\
      1 & otherwise,
    \end{cases*}
  \end{equation}
  for all $1 \leq i \leq \ell$ and $1 \leq j \leq m$.  Thus \autoref{algo:refine} returns $\vec{d},
  \vec{e}$ on input $\vec{d}, \vec{e}$ and \ref{it:9:2} and
  \ref{it:9:1} follow.

Conversely, let $\vec{d} = (d_{1}, \dots, d_{\ell})$, $\vec{e} =
(e_{1}, \dots, e_{m})$, and $\ell = \len(\vec{d}) \geq \len(\vec{e}) =
m$.  We assume that $\vec{d}$ and $\vec{e}$ are not associated and
define $i^{*}$ as the minimal index $1 \leq i^{*} \leq \ell$ such
that there is no injective map $\tau \colon \{1, \dots, i^{*}\} \to
\{1, \dots, m\}$ with
\begin{equation}
  \label{eq:74}
  \begin{split}
    d_{i} & = e_{\tau(i)} \text{ for all } 1 \leq i \leq i^{*}, \\
  \tau(i) & < \tau(j) \text{ for all } 1 \leq i < j \leq i^{*} \text{ with } \gcd(d_{i}, d_{j}) > 1
  \end{split}
\end{equation}
in analogy to \eqref{eq:38} and \eqref{eq:13}, respectively.

We have two possible cases for the execution of the
inner loop, steps~\ref{step:inn_in}-\ref{step:inn_out}, for $i = i^{*}$.
\begin{itemize}
  \item If $c= d_{i^{*}, m+1}^{*}$ in \autoref{step:gcd} for some $j$, then $e_{j, \ell+1}^{*}
    \neq d_{i^{*},m+1}^{*}$ (otherwise, we could extend some injective
    $\tau$ for $i^{*} - 1$ by $i^{*} \mapsto j$) and $e_{j, \ell+1}^{*}$ splits into at least
    two nontrivial factors in steps~\ref{step:split_left} and
    \ref{step:split_right}, thus $\len(\vec{e} \refine \vec{d}) \geq
    \len(\vec{e}) + 1$ and $\vec{e} \refine \vec{d} \neq \vec{e}$.
  \item Otherwise $c \neq d_{i^{*},m+1}^{*}$ in \autoref{step:gcd} for all $j$, and
    $d_{i^{*},m+1}^{*}$ splits into at least two nontrivial factors in
    steps~\ref{step:split_left} and
    \ref{step:split_right},
    thus $\len(\vec{d} \refine \vec{e}) \geq \len(\vec{d}) + 1$ and
    $\vec{d} \refine \vec{e} \neq \vec{d}$.
\end{itemize}
\end{proof}

\begin{proposition}
  \label{pro:17} Let $n$ be a positive integer and $\vec{d}$, $\vec{e}$ two ordered factorizations
  of $n$ with length $\ell$ and $m$, respectively.  Then the following holds.
  \begin{ronumerate}
  \item\label{it:17:1} \autoref{algo:refine} works as specified and requires $O(\ell
    m)$ $\gcd$-computations and $O(\ell m)$ additional integer divisions.
  \item\label{it:17:2} We have $
\pD{n, \{\vec{d},\vec{e}\}} = \pD{n, \{\vec{d} \refine \vec{e},
  \vec{e} \}} = \pD{n, \{\vec{d} \refine \vec{e}, \vec{e} \refine \vec{d}\}}$.
\end{ronumerate}
\end{proposition}

\begin{proof}
  If we ignore the last column of $\vec{d}^{*}$ and $\vec{e}^{*}$,
  respectively, we obtain matrices that are each other's transpose
  before and after every execution of the inner loop in
  \autoref{algo:refine}. We use this property to define a sequence of
  integer matrices
  $\mat^{(k)} \in \ZZ^{(\ell+1) \times (m+1)}$ for $0 \leq k \leq \ell
  m$ to simultaneously capture $\vec{d}^{*}$ and $\vec{e}^{*}$ after
  the inner loop has bee executed $k$ times.

  For $k=0$, let
  \begin{equation}
    \label{eq:44}
\mat^{(0)} = \begin{pmatrix}
1     &  1 & \dots  & 1     & d_{1}    \\
1     & 1 & \dots  & 1     & d_{2}    \\
      & & \vdots &       &          \\
1     & 1 & \dots  & 1     & d_{\ell} \\
e_{1} & e_{2} & \dots  & e_{m} & 1        \\
\end{pmatrix} = (m_{i,j}^{(0)})_{\substack{1 \leq i \leq \ell+1 \\ 1 \leq j
     \leq m+1}}
  \end{equation}
  and for $k = (i-1)m + j > 0$, $1 \leq i \leq \ell$, $1 \leq j \leq m$,
  we define $\mat^{(k)}$ as $\mat^{(k-1)}$ with $m_{i,j}^{(k-1)}$ replaced by
  $c$, $m_{i,m+1}^{(k-1)}$ replaced by $m_{i,m+1}^{(k-1)}/c$, and
  $m_{\ell+1,j}^{(k-1)}$ replaced by $m_{\ell+1,j}^{(k-1)}/c$,
  respectively. Thus, we have the following invariants. For every $0 \leq k
  \leq \ell m$ and every $1 \leq i \leq \ell$, the $i$th row of
  $\mat^{(k)}$ is a factorization of $d_{i}$. Analogously, for every
  $1 \leq j \leq m$, the $j$th column of $\mat^{(k)}$ is a
  factorization of $e_{j}$.

  We have $\vec{d}^{*}$ as $\mat^{(k)}$ with the last row removed, and
  $\vec{e}^{*}$ as $\mat^{(k)}$ with the last column removed and then
  transposed. In particular, the output $\vec{d} \refine \vec{e}$ is
  the first $\ell$ rows of $\mat^{\ell m}$ read as a sequence with
  $1$'s ignored. Analogously, the output $\vec{e} \refine \vec{d}$ is
  the first $m$ columns of $\mat^{\ell m}$ read as a sequence with
  $1$'s ignored.

  \ref{it:17:1} By the invariants of $\mat^{(k)}$ mentioned above, the
  output $\vec{d} \refine \vec{e}$ is a refinement of the input
  $\vec{d}$. Analogously, the output $\vec{e} \refine \vec{d}$ is a
  refinement of the input $\vec{e}$. The outputs also have the same
  basis, namely the entries of $\mat^{(\ell m)}$ different from $1$.

  A
  bijection $\sigma$ on $\{1, \dots, \ell m\}$ is given by $k =
  (i-1)m+j \mapsto (j-1)\ell+i$ with $1\leq i \leq \ell$, $1 \leq j
  \leq m$. And this satisfies \eqref{eq:38} since $d_{k} = c_{i,j} =
  e_{\sigma(k)}$ for all such $i,j$. We also show \eqref{eq:83} for
  $\sigma$. Let $1 \leq k < k' \leq \ell m$ with $k = (i-1)m+j$ and
  $k' = (i'-1)m+j'$. We have to prove, that if $\sigma(k) >
  \sigma(k')$, then $\gcd(d_{k}, d_{k'}) = 1$. The condition is
  equivalent to $i < i'$ and $j > j'$.

  \noindent \textbf{Lemma~3.8a.} \textit{
    Let $1 \leq i \leq m$, $1 \leq j \leq \ell$, $(i - 1)m + j \leq k
    \leq \ell m$ and $c^{(k)}$ the state of \autoref{algo:refine}
    after $k$ executions. Let $R_{i,j} = \prod_{j < j' \leq \ell + 1}
    c_{i,j'}^{(k)}$ and $B_{i,j} = \prod_{i < i' \leq m + 1}
    c_{i',j}^{(k)}$. Then $\gcd(R_{i,j}, B_{i,j}) = 1$. In particular,
    after the algorithm has terminated, we have $\gcd(c_{i,j'},
    c_{i',j}) = 1$ for all $i' > i$, $j' > j$.
}

  \begin{proof}[Proof of Lemma~3.8a]
    Concentrate on the element $c_{i,j'}$.  By construction $\prod_{k
      > j'} c_{i,k}$ and $\prod_{k > i} c_{k,j'}$ are coprime. In
    particular, their factors $c_{i,j}$ and $c_{i',j'}$.
  \end{proof}
  This shows that $\vec{d}^{*}$ and $\vec{e}^{*}$ are associated if we
  restrict $\sigma$ to indices $k$ with $d_{k} > 1$.

  Finally, the only arithmetic costs are the $\gcd$-computations in
  \autoref{step:gcd} and the integer divisions in
  \autoref{step:split_right}.

  \ref{it:17:2} We begin with the first equality. The matrix
  $\vec{d}^{*}$ corresponds to an ordered factorization, when read
  row-by-row and $1$'s ignored. Let $\vec{d}^{(k)}$ correspond to the
  state of the matrix $\vec{d}^{*}$ after the inner loop has been
  executed exactly $k$ times for $0 \leq k \leq \ell m$. Thus
  $\vec{d}^{(0)} = \vec{d}$, $\vec{d}^{(\ell m)} = \vec{d} \refine
  \vec{e}$, and we show inductively
  \begin{equation}
    \label{eq:75} \tag{3.8b}
    \pD{n,  \{\vec{d}^{(k)}, \vec{e}\} } = \pD{n,
    \{\vec{d}^{(k+1)}, \vec{e}\}}
  \end{equation}
  for all $0 \leq k < \ell m$.

  Let $k + 1 = (i - 1)m + j$ with $1 \leq i \leq \ell$, $1 \leq j \leq
  m$. If $c=1$ in \autoref{step:gcd}, then $\vec{d}^{(k+1)} = \vec{d}^{(k)}$ and \eqref{eq:75} holds
  trivially. Otherwise $c > 1$, and $\vec{d}^{(k+1)}$ is the proper
  refinement
  of $\vec{d}^{(k)}$, where the entry $d_{i,m+1}^{*}$ in
  $\vec{d}^{(k)}$ is replaced by
  the pair $(c, d_{i, m+1}^{*}/c)$.

  We have to show that if a polynomial has decomposition degree
  sequences $\vec{d}^{(k-1)}$ and $\vec{e}$, then it also has
  decomposition degree sequence $\vec{d}^{(k)}$. This follows from the
  following generalization of \autoref{pro:tortrat}.

  \noindent \textbf{Lemma~3.8c.} \textit{
  Let $g_{1} \circ g_{2} \circ \dots \circ g_{\ell} = h_{1} \circ
  h_{2} \circ \dots \circ h_{m}$ be two decompositions with degree
  sequence $\vec{d}$ and $\vec{e}$, respectively.  Let $1 \leq i \leq
  \ell$, $1 \leq j \leq m$, $c = \gcd(d_{i}, e_{j})$, and
  \begin{equation}
    \gcd(d_{1}\cdot \dots \cdot d_{i-1} \cdot d_{i}, e_{1} \cdot \dots \cdot
    e_{j-1}) = \gcd(d_{1} \cdot \dots \cdot d_{i-1}, e_{1} \cdot \dots
  \cdot  e_{j-1} \cdot e_{j}). \label{eq:79} \tag{3.8d}
  \end{equation}
  Then there are unique monic original polynomials $u$ and $v$ of
  degree $c$ and $d_{i}/c$, respectively, such that
  \begin{equation}
    \label{eq:80} \tag{3.8e}
    g_{i} = u \circ v.
  \end{equation}
  Therefore, if a monic original polynomial $f$ has decomposition degree sequences $\vec{d}$ and
  $\vec{e}$, then it also has decomposition degree sequence
  $\vec{d}^{*} = (d_{1}, \dots, d_{i-1}, c, d_{i}/c, d_{i+1}, \dots, d_{\ell})$.
  }

  \begin{proof}[Proof of Lemma~3.8c]
  Let $A = g_{1} \circ \dots \circ g_{i-1}$,
  $B = h_{1} \circ \dots \circ h_{j-1}$, and $b = \gcd(\deg(A),
  \deg(B))$. Then \eqref{eq:79} reads $\gcd(\deg(A\circ g_{i}),
  \deg(B)) = \gcd(\deg(A), \deg(B\circ h_{j}))$.  This implies
  $\gcd(\deg(B)/\gcd(\deg(A), \deg(B)), \deg(g_{i})) =
  \gcd(\deg(A)/\gcd(\deg(A), \deg(B)), \deg(h_{j}))$ and since the
  first arguments of both outer $\gcd$'s are coprime, this quantity is
  $1$. This proves
  \begin{align}
    \gcd(\deg(A \circ g_{i}), \deg(B \circ h_{j})) & = \gcd(\deg(A),
    \deg(B)) \cdot \gcd(\deg(g), \deg(h)) \\
    & \quad \cdot \gcd(\frac{\deg(A)}{\gcd(\deg(A), \deg(B))},
    \frac{\deg(h_{j})}{\gcd(\deg(g_{i}, h_{j}))}) \\
    & \quad \cdot \gcd(\frac{\deg(B)}{\gcd(\deg(A), \deg(B))},
    \frac{\deg(g_{i})}{\gcd(\deg(g_{i}, h_{j}))}) \\
    & = \gcd(\deg(A), \deg(B)) \cdot \gcd(\deg(g_{i}), \deg(h_{j})) \\
    & = b c.
    \label{eq:97} \tag{3.8f}
  \end{align}
Then
  \autoref{pro:tortrat} applied to left components of the bi-decompositions
  \begin{equation}
    \label{eq:77} \tag{3.8g}
    A \circ ( g_{i} \circ \dots \circ g_{\ell}) = B \circ (h_{j} \circ
    \dots \circ h_{m})
  \end{equation}
  guarantees the existence of unique, monic original $C, A', B'$ with $\deg(C)
  = b$ and $\gcd(\deg(A'), \deg(B'))=1$, such that
  \begin{equation}
    \label{eq:78} \tag{3.8h}
    A = C \circ A' \text{ and } B = C \circ B'.
  \end{equation}
  We substitute \eqref{eq:78} back into \eqref{eq:77}, ignore the
  common left component $C$ due to the absence of equal-degree
  collisions, and write with the associativity of composition
  \begin{equation}
    (A' \circ  g_{i}) \circ (g_{i+1} \circ \dots \circ g_{\ell}) = (B' \circ h_{j}) \circ
    (h_{j+1} \circ \dots \circ h_{m}).
  \end{equation}
  From \eqref{eq:97}, we have $\gcd(\deg(A' \circ g_{i}), \deg(B'
  \circ h_{j})) = \gcd(d_{i}, e_{j}) = c$. With \autoref{pro:tortrat},
  we obtain some monic original $w$ and $A''$ of degree $c$ and
  $\deg(A' \circ g_{i})/c = \deg(A') \cdot d _{i}/c$, respectively,
  such that
  \begin{equation}
    \label{eq:81} \tag{3.8i}
    A' \circ g_{i} = w \circ A''.
  \end{equation}
  We have $\gcd(\deg(g_{i}), \deg(A'')) = d_{i}/c$ and a final
  application of \autoref{pro:tortrat} to the right components of \eqref{eq:81}
  provides the decomposition for $g$, claimed in \eqref{eq:80}.
  \end{proof}

  To apply this result with $\vec{d} = \vec{d}^{(k)}$ and $\vec{e}$,
  we have to provide \eqref{eq:79}. For $(k+1) = (i-1)m + j$, we split
  $D = d_{1} \cdot \dots \cdot d_{i-1}$ and $E = e_{1} \cdot \dots
  \cdot e_{j-1}$ into their common (left upper subset) $C$ and the
  remainders $R$ and $B$, respectively. By Lemma~3.8a, we have
  $\gcd(R,B) = 1$ and therefore $\gcd(D,E) = C$. The same lemma shows
  $\gcd(d_{i}, B) = 1 = \gcd(R, e_{j})$ and we have
  \begin{equation}
    \label{eq:43}
    \gcd(D d_{i}, E) = C \gcd(d_{i}, B) = C \gcd(R, e_{j}) = \gcd(D, E
    e_{j}),
  \end{equation}
  as required for \eqref{eq:75}.

  Finally, interchanging the r{\^o}les of $\vec{d}$ and $\vec{e}$ yields
  \begin{equation}
    \label{eq:89}
    \pD{n, \{\vec{d}, \vec{e}\}} = \pD{n, \{\vec{d} \refine \vec{e},
      \vec{e}\}} = \pD{n, \{\vec{d}, \vec{e} \refine \vec{d}\}} =
    \pD{n, \{\vec{d} \refine \vec{e}, \vec{e} \refine \vec{d},
      \vec{d}, \vec{e}\}} = \pD{n, \{\vec{d} \refine \vec{e}, \vec{e} \refine \vec{d}\}},
  \end{equation}
  since composition degree sequence $\vec{d} \refine \vec{e}$ implies
  $\vec{d}$ and similarly $\vec{e} \refine \vec{d}$ implies
  $\vec{e}$.
\end{proof}

\begin{example}
  \label{exa:14}
  \begin{sagesilent}
    k = 7
    n = factorial(k)
    dd = [12, 420]
    ee = [14, 360]
    assert prod(dd) == n and prod(ee) == n
  \end{sagesilent}

Let $n = \sage{k} ! = \sage{n}$, $\vec{d} = \sage{tuple(dd)}$, and $\vec{e} = \sage{tuple(ee)}$.  We have as refinements
\begin{equation}
\label{eq:16}
  \begin{split}
    \vec{d} \refine \vec{e} & = \sage{tuple(refine(dd, ee))}, \\
    \vec{e} \refine \vec{d} & = \sage{tuple(refine(ee, dd))},
  \end{split}
\end{equation}
and any $f \in \pD{n, \{\vec{d}, \vec{e}\}}$ has a unique decomposition
  $ f = a \circ g \circ b$ with $a \in \pP{2}$, $g \in \pD{42,
    \{(6,7),(7,6)\}}$, and $b \in \pP{60}$ by \autoref{pro:tortrat}.
\end{example}

Given a set $\vec{D}$ with more than two ordered factorizations, we
repeatedly replace pairs $\vec{d}, \vec{e} \in \vec{D}$ by $\vec{d}
\refine \vec{e}$ and $\vec{e} \refine \vec{d}$, respectively, until we
reach a refinement $\vec{D}^{*}$ invariant under this operation.  This
process terminates by \autoref{lem:9}.  The
result depends on the order of the applied refinements, but any order
ensures the desired properties described by the following proposition.

\begin{proposition}
  \label{pro:7}
  Let $n$ be a positive integer and $\vec{D}$ a set of $c$ ordered
  factorizations of $n$.  There is a set $\vec{D}^{*}$ of at most $c$ ordered
  factorizations of $n$ with the following properties.
\begin{ronumerate}
\item\label{it:7:1} All ordered factorizations of $\vec{D}^{*}$ are pairwise associated.
\item\label{it:7:2} $\pD{n, \vec{D^{*}}} = \pD{n, \vec{D}}$.
\item $\vec{D}^{*}$ can be computed from $\vec{D}$ with at most
  $O(c^{2})$ calls to \autoref{algo:refine}.
\end{ronumerate}
\end{proposition}

\begin{proof}
  For $c = 1$, we have $\vec{D} = \{d\}$ and $\vec{D}^{*} = \{d\}$
  satisfies all claims.

  For $c = 2$, we have $\vec{D} = \{\vec{d}, \vec{e}\}$ for ordered
  factorizations $\vec{d} \neq \vec{e}$, and $\vec{D}^{*} = \{\vec{d}
  \refine \vec{e}, \vec{e} \refine \vec{d}\}$ satisfies all claims by
  \autoref{pro:17}.

  Let $c > 2$ and $\vec{D} = \{\vec{d}^{(1)}, \dots,
  \vec{d}^{c}\}$. By induction assumption, we can assume all $\vec{d}^{(i)}$ for
  $1 \leq i < c-1$ be pairwise associated. Let $\vec{d}^{(c)} =
  \vec{f}$ and $\vec{D}^{*} = \{\vec{d}^{(1)} \refine \vec{f},
  \vec{d}^{(2)} \refine \vec{f}, \dots, \vec{d}^{(c-1)} \refine
  \vec{f}, \vec{f} \refine \vec{d}^{(1)} \}$. Clearly $\pD{n, \vec{D}}
  = \pP{n, \vec{D}^{*}}$ and it remains to show that all elements of
  $\vec{D}^{*}$ are pairwise associated.

  By construction, we have $\vec{d}^{(1)} \refine \vec{f}$ associated
  with $\vec{f} \refine \vec{d}^{(1)}$ and by transitivity of
  associatedness the following
  lemma suffices.

  \noindent \textbf{Lemma~3.11a.} \textit{
    If $\vec{d}^{*}$ and $\vec{e}^{*}$ are associated, then so are
    $\vec{d}^{*} \refine \vec{f}$ and $\vec{e}^{*} \refine \vec{f}$
    for any factorization $\vec{f}$.
  }

  \begin{proof}
    Let $\sigma = \sigma(\vec{d}^{*}, \vec{e}^{*})$ and compare the
    matrices
    \begin{align}
      \mat = \mat(\vec{d}^{*}, \vec{f}) \text{ and } \nat =
      \mat(\vec{e}^{*}, \vec{f}).
    \end{align}
    The claimed bijection between the indices of $\mat$ and $\nat$ is given by
    mapping row $i$ to row $\sigma(i)$ (followed by identity on the
    columns).

    Assume for contradiction that $i$ is the minimal row index such
    that $\mat_{i, *} \neq \nat_{\sigma(i), *}$ and $j$ is the minimal
    column index such that $\mat_{i,j} \neq \nat_{\sigma(i), j}$.

    Let $\nat_{\sigma(i),j} = a \mat_{i,j}$ with $a > 1$. Then there
    is a column $j' > j$, such that $a \mid \mat_{i,j'}$, since the
    rows $\mat_{i,*}$ and $\nat_{\sigma(i),*}$ are both factorizations
    of $d_{i} = e_{\sigma(i)}$. Also there is a row $i' > i$, such
    that $a \mid \mat_{i',j}$, since the earlier occurrences of $a$ in
    that column are pairwise matched.

    By Lemma~3.8a, this is a contradiction. And analogously, if
    $\mat_{i,j} = a \nat_{\sigma(i),j}$ with $a > 1$.
  \end{proof}
\end{proof}

Any $\vec{D}^{*}$ satisfying \autoref{pro:7}\ref{it:7:1}-\ref{it:7:2} is
called a \emph{normalization} of $\vec{D}$.  For a normalized $\vec{D}
= \{ \vec{d}^{(k)} \colon 1 \leq k \leq c$, we have the same basis
$\uset{\vec{d}^{(k)}}$ for all $1 \leq k \leq c$ and call this multiset
the \emph{basis} of $\vec{D}$, denoted by $\uset{\vec{D}}$.

\begin{example}
  \label{exa:9}
  \begin{sagesilent}
    ff = [20, 252]
  \end{sagesilent}

  We add the ordered factorization $\vec{f} = \sage{tuple(ff)}$ to $\vec{D} =
  \{ \vec{d}, \vec{e} \}$ of \autoref{exa:14} and obtain from
  \eqref{eq:16} through refinement with $\vec{f}$
\begin{equation}
\label{eq:6} \tag{3.12b}
  \begin{split}
    \vec{d}^{*} = (\vec{d} \refine \vec{e}) \refine \vec{f} & =
    \sage{tuple(refine(refine(dd,ee),ff))}, \\
    \vec{e}^{*} = (\vec{e} \refine \vec{d}) \refine \vec{f} & =
    \sage{tuple(refine(refine(ee,dd),ff))}, \\
    \vec{f}^{*} = (\vec{f} \refine \vec{d}) \refine \vec{e} & =
    \sage{tuple(refine(refine(ff,dd),ee))}.
  \end{split}
\end{equation}
  Any $f \in \pP{n, \{\vec{d}, \vec{e}, \vec{f} \}} = \pP{n, \{\vec{d}^{*}, \vec{e}^{*}, \vec{f}^{*}\}}$ has a unique
  decomposition $f = a \circ g \circ b$ with $a \in \pP{2}$, $g \in
  \pD{210, \{(2,3,7,5), (7,2,3,5), (2,5,3,7)\}}$, and $b \in \pP{12}$.
  The normalized set $\{\vec{d}^{*}, \vec{e}^{*}, \vec{f}^{*}\}$ has
  basis $\{2,2,3,5,7,12\}$.
\end{example}

\subsection{The relation graph of \large $\vec{D}$}
\label{sec:size-comp-copr}

An ordered factorization $\vec{d} = (d_{1}, d_{2}, \dots, d_{\ell})$
defines a relation $\prec_{\vec{d}}$ on its basis $\uset{\vec{d}} =
\{d_{1}, d_{2}, \dots, d_{\ell} \}$ by
\begin{equation}
  \label{eq:17}
  d_{i} \prec_{\vec{d}} d_{j} \text{ for } 1 \leq i < j \leq \ell.
\end{equation}
In other words, $d_{i} \prec_{\vec{d}} d_{j}$ if $d_{i}$ appears
before $d_{j}$ in the ordered factorization $\vec{d}$, where we
distinguish between repeated factors in the multiset $\uset{\vec{d}}$.  We define the
\emph{relation graph} $G_{\vec{d}}$ as directed graph with
\begin{itemize}
\item vertices $\uset{\vec{d}} =
\{d_{1}, d_{2}, \dots, d_{\ell} \}$ and
\item directed edges $(d_{j}, d_{i}) = d_{i} \gets d_{j}$ for $d_{i} \prec_{\vec{d}} d_{j}$.
\end{itemize}
This graph is a \emph{transitive tournament}, that is a complete graph
with directed edges, where a path $d \gets e \gets f$ implies an edge $d \gets f$ for any
vertices $d,e,f \in \uset{\vec{d}}$.

Now, let $\vec{D} = \{\vec{d}^{(1)}, \vec{d}^{(2)}, \dots, \vec{d}^{(c)}\}$
be a normalized set of $c$ ordered factorizations with common basis $\uset{\vec{D}} = \{d_{1}, d_{2}, \dots,
d_{\ell} \}$.  The relation $\prec_{\vec{D}}$ is the union of the
relations $\prec_{\vec{d}^{(k)}}$ for $1 \leq k \leq c$ and the \emph{relation graph} $G_{\vec{D}}$ is the union of the
relation graphs $G_{\vec{d}^{(k)}}$ for $1 \leq k \leq c$.  The
undirected graph underlying $G_{\vec{D}}$ is still complete, but may be
intransitive.  See \autoref{fig:rel-graph} for the relation graphs of
\autoref{exa:14} and \autoref{exa:9}.

We can express the relation $\prec_{\vec{D}}$ with the permutations
\eqref{eq:38}.  Let $\sigma_{k} = \sigma(\vec{d}^{(1)},
\vec{d}^{(k)})$ for $1 \leq k \leq c$.  Then $\sigma_{1}$ is the
identity on $1, 2, \dots, \ell$ and we have
\begin{equation}
  \label{eq:39}
  d_{i} \prec_{\vec{d}^{(k)}} d_{j}
\end{equation}
if and only if $\sigma_{k}(i) < \sigma_{k}(j)$.

\begin{sagesilent}
DD1 = refined_closure([dd, ee])
G1, dd1 = graph_from_relations(DD1)
G1.relabel(dd1)

DD2 = refined_closure([dd, ee, ff])
G2, dd2 = graph_from_relations(DD2)
dd2[0] = r"2'"
G2.relabel(dd2)

vert_pos = G2.layout_circular()

angle = N(pi/3)
rot_matrix = matrix([[cos(angle), -sin(angle)],[sin(angle), cos(angle)]])

# Hopefully, the positions that work perfectly!
new_pos = {}
for i in vert_pos:
    new_pos[i] = tuple(rot_matrix * vector(vert_pos[i]))
\end{sagesilent}

\begin{figure*}
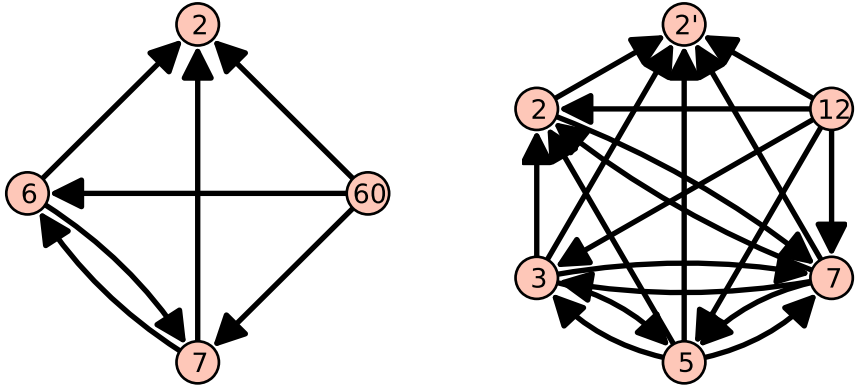

  \centering
\sageplot{G1.plot(layout='circular',
  vertex_size=1000, figsize=[2.3,2.3])}
\sageplot{G2.plot(
  vertex_size=1000, pos=new_pos, figsize=[2.3,2.3])}
  \caption{Relation graphs of (\ref{eq:16}) and (\ref{eq:6}); in the
    latter, $2'$ denotes the first $2$ in each ordered factorization.}
  \label{fig:rel-graph}
\end{figure*}
\noeqref{eq:6}

A Hamiltonian path $\vec{e} = e_{1} \gets \dots \gets e_{\ell}$ in a graph $G$ visits each vertex exactly
once. We call $\vec{e}$ \emph{transitive}, if its transitive closure is a
subgraph of $G$.  In other words, $\vec{e}$ is transitive if $e_{i}
\gets e_{j}$ is an edge in $G$ for all $1 \leq i < j \leq \ell$. For a relation graph $G$ with vertices $d_{1},d_{2}, \dots, d_{\ell}$ and
$n = \prod_{1 \leq i \leq \ell} d_{i}$, we define
\begin{equation}
  \label{eq:19}
\begin{split}
  \pD{G} & = \{ f \in \pP{n} \colon \text{ for every transitive Hamiltonian path } e_{1} \gets \dots \gets e_{\ell} \\
& \quad \text{ in $G$, there is a decomposition } f = g_{1} \circ g_{2} \circ \dots
  \circ g_{\ell} \\
& \quad \text{ with } \deg g_{i} = e_{i} \text{ for } 1 \leq
  i \leq \ell\}.
\end{split}
\end{equation}

If $G = \{d\}$ is a singleton, we have $\pD{G}= \pP{d}$.

\addtocounter{equation}{-1}

\begin{proposition}
  \label{pro:13}
  Let $n$ be a positive integer, $\vec{D}$ a normalized set of ordered
  factorizations of $n$, and $G$ the relation graph of $\vec{D}$.
  We have
  \begin{equation}
    \label{eq:4}
    \pD{n,\vec{D}} = \pD{G}.
  \end{equation}
\end{proposition}

\begin{proof}
  Every transitive tournament $G_{\vec{d}}$ for $\vec{d} \in \vec{D}$,
  has $\vec{d}$ as its unique transitive Hamiltonian path. Since $G$
  is the union of all such $G_{\vec{d}}$, we
  have ``$\supseteq$''.

  For ``$\subseteq$'', we have to show that every polynomial with
  decomposition degree sequences $\vec{D}$ also has decomposition
  degree sequence $\vec{d}^{*}$ for every transitive Hamiltonian path
  $\vec{d}^{*}$ in $G$. We proceed on two levels. First, we derive all
  transitive Hamiltonian paths in $G$ from ``twisting'' the paths
  given by $\vec{D}$. Second, we show that the corresponding ``twisted''
  decomposition degree sequences follow from the given ones.

  Let $\vec{d}^{*}$ be a transitive Hamiltonian path in $G$ and
  $\vec{d} \in \vec{D}$ arbitrary. We use \textsc{Bubble-Sort} to
  transform $\vec{d}$ into $\vec{d}^{*}$ and call the intermediate
  states after $k$ passes $\vec{d}^{(k)}$, $0 \leq k \leq c$, such
  that $\vec{d}^{(0)} = \vec{d}$ and $\vec{d}^{(c)} = \vec{d}^{*}$.

  \begin{algorithm2e}[H]
    \caption{\textsc{Bubble-Sort} $\vec{d}$ according to $\vec{d}^{*}$}
    \label{algo:bubble}

  $\ell \gets \len(\vec{d})$\;
  $k \gets 0$, $\vec{d}^{(0)} \gets \vec{d}$\;
  \While{$\vec{d}^{(k)} \neq \vec{d}^{*}$}{
    $k \gets k+1$, $\vec{d}^{(k)} \gets \vec{d}^{(k-1)}$ \tcc{copy
      previous state}
    \For{$i = 1, \dots, \ell - 1$}{
      $\sigma = \sigma(\vec{d}^{(k)}, \vec{d}^{*})$ \label{step:bubble:sigma}\;
      \If{$\sigma(i) > \sigma(i+1)$}{
        $(d_{i}^{(k)}, d_{i+1}^{(k)}) \gets (d_{i+1}^{(k)},
        d_{i}^{(k)})$ \label{step:bubble:swap} \tcc{swap}
      }
    }
  }
  $c \gets k$\;

  \end{algorithm2e}
\addtocounter{equation}{-1}

  In other words, $\vec{d}^{(k)}$ is obtained from $\vec{d}^{(k-1)}$
  by at most $\ell - 1$ ``swaps'' of adjacent vertices. \autoref{fig:swap}
  visualizes a swap of $d_{i}^{(k)}$ and
  $d_{i+1}^{(k)}$ as in \autoref{step:bubble:swap}. The fundamental
  properties of \textsc{Bubble-Sort} guarantee correctness and $c \leq \ell
  (\ell-1)/2$, see \citet[Problem~2.2]{corlei09}.

    \begin{sagesilent}
g3 = {1: [0, 2], 2: [1], 3: [2]}
g4 = {1: [2], 2: [0, 1], 3: [1]}

G3 = DiGraph(g3, multiedges=True, implementation='c_graph', vertex_labels=True)
G4 = DiGraph(g4, multiedges=True, implementation='c_graph',
vertex_labels=True)

dd = [r"$d_{i-1}$", r"$d_{i}$", r"$d_{i+1}$", r"$d_{i+2}$"]
G3.relabel(dd)
G4.relabel(dd)

vert_pos = G3.layout_circular()
new_pos = {}

angle = N(0*pi/4)
rot_matrix = matrix([[cos(angle), -sin(angle)],[sin(angle), cos(angle)]])
stretch = 1

new_pos[r"$d_{i-1}$"] = tuple(rot_matrix * vector((stretch*(-2), stretch*0)))
new_pos[r"$d_{i}$"] = tuple(rot_matrix * vector((stretch*0, stretch*1)))
new_pos[r"$d_{i+1}$"] = tuple(rot_matrix * vector((stretch*0, stretch*(-1))))
new_pos[r"$d_{i+2}$"] = tuple(rot_matrix * vector((stretch*2, stretch*0)))
    \end{sagesilent}

\begin{figure*}
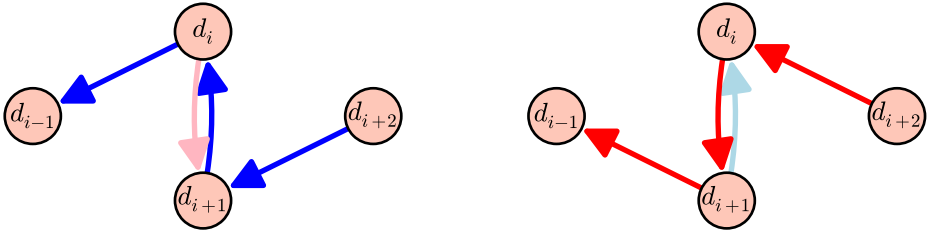

  \centering
    \sageplot[
    ]{G3.plot(
    edge_colors={'blue':[(dd[3],dd[2],None),(dd[2],dd[1],None),(dd[1],dd[0],None)], 'lightpink':[(dd[1],dd[2],None)]},
    vertex_size=7000,
    pos=new_pos,
    figsize=[2.3, 2.3])}
    \sageplot[
    ]{G4.plot(
    edge_colors={'red':[(dd[3],dd[1],None),(dd[1],dd[2],None),(dd[2],dd[0],None)], 'lightblue':[(dd[2],dd[1],None)]},
    vertex_size=7000,
    pos=new_pos,
    figsize=[2.3, 2.3])}
  \caption{A ``swap'' between two transitive Hamiltonian paths $d_{i-1} \gets
    d_{i} \gets d_{i+1} \gets d_{i+2}$ and $d_{i-1} \gets
    d_{i} \gets d_{i+1} \gets d_{i+2}$ along the bidirectional edge
    between $d_{i}$ and $d_{i+1}$.}
  \label{fig:swap}
\end{figure*}

  Furthermore, the following holds.
  \begin{ronumerate}
  \item\label{it:bu:1} Every pair $(d_{i}^{(k)}, d_{i+1}^{(k)})$ of
    swapped vertices in \autoref{step:bubble:swap} is connected by a bidirectional edge in $G$.
  \item\label{it:bu:2} Every $\vec{d}^{(k)}$, $0 \leq k \leq c$, is a
    transitive Hamiltonian path in $G$.
  \end{ronumerate}
  For \ref{it:bu:1}, we have the edge $d_{i}^{(k)} \gets d_{i+1}^{(k)}$ from
  $\vec{d}^{(k-1)}$ and the edge $d_{\sigma(i+1)}^{*} = d_{i+1}^{(k)} \gets
  d_{i}^{(k)} = d_{\sigma(i)}^{*}$ from $\vec{d}^{*}$ with $\sigma$ as
  in \autoref{step:bubble:sigma}.

  For $k=0$, \ref{it:bu:2} holds by definition. For $k > 0$ it
  follows inductively from $k - 1$, since a swap merely replaces the 4-subpath $d_{i-1} \gets
    d_{i} \gets d_{i+1} \gets d_{i+2}$ by $d_{i-1} \gets d_{i+1} \gets
    d_{i} \gets d_{i+2}$, where the outer edges are guaranteed in $G$
    by transitivity of $\vec{d}^{(k-1)}$ and the inner edge by
    \ref{it:bu:1}. Thus, the swapped path is also a transitive
    Hamiltonian path in $G$.

  Now, we mirror the ``swaps'' of vertices by ``Ritt moves'' of
  components as introduced by \cite{ziemue08}.

\noindent \textbf{Claim~3.13a} (Ritt moves). \textit{
    Let $g_{1} \circ \dots \circ g_{\ell} = h_{1} \circ \dots \circ
    h_{\ell}$ be decompositions with degree sequence $\vec{d}$ and
    $\vec{e}$, respectively.  Let $\vec{d}$ and $\vec{e}$ be
    associated, $\sigma = \sigma(\vec{d}, \vec{e})$, and $1 \leq i <
    \ell$ with $\sigma(i) > \sigma(i+1)$. Then
    \begin{equation}
      \label{eq:94}
      g_{i} \circ g_{i+1} = g_{i}^{*} \circ g_{i+1}^{*}
    \end{equation}
    with $\deg(g_{i}) = \deg(g_{i+1}^{*})$ and $\deg(g_{i+1})
    = \deg(g_{i}^{*})$. Therefore, if some monic original polynomial
    $f$ has decomposition degree sequences $\vec{d}$ and $\vec{e}$,
    it also has the decomposition degree sequence $\vec{d}^{*} = (d_{1}, \dots, d_{i-1}, d_{i+1},
    d_{i}, d_{i+2}, \dots, d_{\ell})$.
  }

  The claim is based on the following lemma.

  \noindent \textbf{Lemma~3.13b.} \textit{
    \label{lem:2}
    Let $\vec{d}$ and $\vec{e}$ be associated ordered factorizations,
    $\sigma = \sigma(\vec{d}, \vec{e})$, $1 \leq i \leq
    \len(\vec{d})$, and $j = \sigma(i)$. Then
    \begin{align}
      \label{eq:66}
      \gcd(d_{1} \cdot \dots \cdot d_{i-1}, e_{1} \cdot \dots \cdot
      e_{j-1}) & = \gcd(d_{1}\cdot \dots \cdot d_{i-1} \cdot d_{i}, e_{1}\cdot \dots
      \cdot e_{j-1}) \\
      & = \gcd(d_{1} \cdot \dots \cdot d_{i-1}, e_{1} \cdot \dots
      \cdot e_{j-1} \cdot e_{j}).
    \end{align}
    In particular, \eqref{eq:79} holds.
  }

  \begin{proof}[Proof of Lemma~3.13b]
    For any $1 \leq k < j$, with
    $\gcd(e_{k}, e_{j}) = \gcd(e_{k}, d_{i}) > 1$, we have
    $\sigma^{-1}(k) < i$ due to \eqref{eq:13}. In other words, $\sigma^{-1}$ maps all
    indices $1 \leq k < j$, where  $\gcd(e_{k}, d_{i}) > 1$, into the set $\{1,
    \dots, i-1\}$. Therefore
    \begin{align}
      \label{eq:69}
      \gcd(\frac{e_{1}\cdot \dots \cdot e_{j-1}}{\gcd(d_{1} \cdot
        \dots \cdot d_{i-1}, e_{1}\cdot \dots \cdot e_{j-1})}, d_{i})
      & = 1, \\
      \gcd(d_{1}\cdot \dots \cdot d_{i-1} \cdot d_{i}, e_{1}\cdot
      \dots \cdot e_{j-1}) & = \gcd(\gcd(d_{1} \cdot \dots \cdot
      d_{i-1}, e_{1}\cdot \dots \cdot e_{j-1}) d_{i},
      e_{1}\cdot \dots \cdot e_{j-1}) \\
      & = \gcd(d_{1} \cdot \dots \cdot d_{i-1}, e_{1}\cdot \dots \cdot
      e_{j-1}).
    \end{align}
  \end{proof}

  Let $j' = \sigma(i+1) < \sigma(i) = j$, $A = g_{1} \circ \dots \circ
  g_{i-1}$, $C = g_{i+2} \circ \dots \circ g_{\ell}$, $A' = h_{1}
  \circ \dots \circ h_{j'-1}$, $B' = h_{j'+1} \circ \dots \circ
  h_{j-1}$, and $C' = h_{j+1} \circ \dots \circ h_{\ell}$, such that
  \begin{equation}
    \label{eq:85}
      A \circ g_{i} \circ g_{i+1} \circ C  = A' \circ h_{j'} \circ B'
      \circ h_{j} \circ C'.
  \end{equation}
  Lemma~3.13b for $i$ and $i+1$ yields
  \begin{align}
    \gcd(\deg(A \circ g_{i}), \deg(A' \circ h_{j'} \circ B')) & = \gcd(\deg(A),
    \deg(A' \circ h_{j'} \circ B')), \\
    \gcd(\deg(A \circ g_{i} \circ g_{i+1}), \deg(A')) & = \gcd(\deg(A
    \circ g_{i}), \deg(A' \circ h_{j'})),
  \end{align}
  respectively. From the former, we derive
  \begin{align}
    \label{eq:71}
    1 & = \gcd(g_{i}, \frac{\deg(A' \circ h_{j'} \circ B')}{\gcd(\deg(A),
      \deg(A' \circ h_{j'} \circ B'))}) \\
      & = \gcd(g_{i}, \frac{\deg(A'
      \circ h_{j})}{\gcd(\deg(A),
      \deg(A' \circ h_{j'}))}).
  \end{align}
  And then continue the latter as
  \begin{align}
& \gcd(\deg(A \circ g_{i} \circ g_{i+1}), \deg(A')) \\
& = \gcd(\deg(A \circ g_{i}), \deg(A' \circ h_{j'})) \\
&  = \gcd(\deg(A), \deg(A' \circ h_{j'})) \cdot \gcd(g_{i}, \frac{\deg(A'
      \circ h_{j'})}{\gcd(\deg(A),
      \deg(A' \circ h_{j'}))}) \\
   & = \gcd(\deg(A), \deg(A' \circ h_{j'})). \label{eq:98} \tag{3.13c}
  \end{align}
  Let $G = g_{i} \circ g_{i+1}$ and $H = h_{j}$. We have $\gcd(d_{i},
  d_{i+1}) = 1$ due the ``twisting condition'' $\sigma(i+1) <
  \sigma(i)$  and therefore $\gcd( \deg(G), \deg(H)) = d_{i+1}$. We
  apply Lemma~3.8c with $g_{i} = G$, $h_{j} = H$, and $c =
  d_{i+1}$ in the notation of that claim, and find, since \eqref{eq:98} provides
  condition \eqref{eq:79},
  \begin{equation}
    \label{eq:72}
    G = g_{i}^{*} \circ g_{i+1}^{*}
  \end{equation}
  with $\deg(g_{i}^{*}) = d_{i+1}$ and $\deg(g_{i+1}^{*}) = d_{i}$ as
  required.

Repeated application of Claim~3.13a shows that for every $f \in
\pD{n, \vec{D}}$, $\vec{d} \in \vec{D}$, and every transitive Hamiltonian path $\vec{d}^{*}$
in $G$, we have $\vec{d}^{(k)}$ as in \autoref{algo:bubble} as
decomposition degree sequence. In particular, $\vec{d}^{(c)} = \vec{d}^{*}$.
\end{proof}

\subsection{The Decomposition of \large $\pD{n, \vec{D}}$}

Every directed graph admits a decomposition into strictly connected
components, where any two distinct vertices are connected by paths in
either direction.  Since a relation graph $G$ is the union of directed
complete graphs, its strictly connected components $G_{i}$, $1 \leq i
\leq \ell$, are again relation graphs and form a chain $G_{1} \gets
G_{2} \gets \dots \gets G_{\ell}$. \autoref{fig:conn_comp} shows the
connected components of the relation graphs \autoref{fig:rel-graph}

    \begin{sagesilent}
g3 = {2: [], 6: [7], 7: [6], 60: []}
g4 = {1: [], 2: [7], 3: [2, 5, 7], 5: [2, 3, 7], 7: [2, 3, 5], 12: []}

G3 = DiGraph(g3, multiedges=True, implementation='c_graph', vertex_labels=True)
G4 = DiGraph(g4, multiedges=True, implementation='c_graph', vertex_labels=True)
G4.relabel({1: r"2'"})

angle = N(0*pi/4)
rot_matrix = matrix([[cos(angle), -sin(angle)],[sin(angle), cos(angle)]])

new_pos3 = {}
new_pos3[2]  = (-1, 0)
new_pos3[6]  = ( 0, 1)
new_pos3[7]  = ( 0,-1)
new_pos3[60] = ( 1, 0)

new_pos4 = {}
new_pos4["2'"] = (-2, 0)
new_pos4[2]    = (-1, 0)
new_pos4[3]    = ( 0, 1)
new_pos4[5]    = ( 0,-1)
new_pos4[7]    = ( 1, 0)
new_pos4[12]   = ( 2, 0)
    \end{sagesilent}

\begin{figure*}
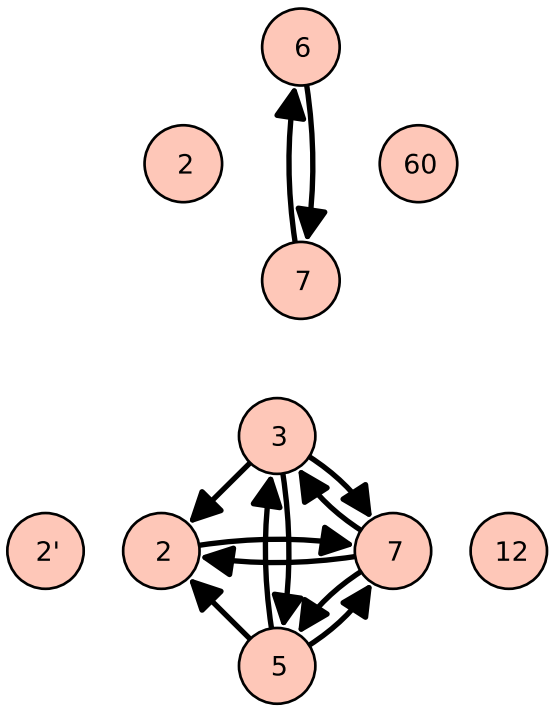

  \centering
    \sageplot[
    ]{G3.plot(
    vertex_size=7000,
    pos=new_pos3,
    figsize=[1.7, 1.7])}
    \sageplot[
    ]{G4.plot(
    vertex_size=7000,
    pos=new_pos4,
    figsize=[3, 2])}
  \caption{The three strongly connected components of each relation graph in
    \autoref{fig:rel-graph}, respectively.}
  \label{fig:conn_comp}
\end{figure*}

\begin{theorem}
  \label{thm:split} Let $G$ be a relation graph with strongly connected
  components $G_{1} \gets G_{2} \gets \dots \gets G_{\ell}$.  We have
  \begin{equation}
    \label{eq:11}
    \pD{G} = \pD{G_{1}} \circ \pD{G_{2}} \circ \dots \circ \pD{G_{\ell}}
  \end{equation}
  and for any $f \in \pD{G}$, we have uniquely determined $g_{i} \in
  \pD{G_{i}}$ such that $f = g_{1} \circ g_{2} \circ \dots \circ
  g_{\ell}$.  Furthermore, over a finite field $F = \Fq$ with $q$
  elements, we have
  \begin{equation}
    \label{eq:12}
    \# \pD{G} = \prod_{1 \leq i \leq \ell} \# \pD{G_{i}}.
  \end{equation}
\end{theorem}

\begin{proof}
  For $f \in \pP{n}$, where $n = \prod_{v \in G} v$, we show that the
  following are equivalent.
  \begin{ronumerate}
    \item\label{it:split:1} The polynomial $f$ has decomposition
      degree sequence $\vec{d}$ for every transitive Hamiltonian path
      $\vec{d}$ in G.
    \item\label{it:split:2} The polynomial $f$ has decomposition
      degree sequence $\vec{d} =
      \vec{d}_{1} \gets \vec{d}_{2} \gets \dots \gets \vec{d}_{\ell}$
      for every concatenation of transitive Hamiltonian paths
      $\vec{d}_{i}$ in $G_{i}$ for $1
      \leq i \leq \ell$.
  \end{ronumerate}

  Assume \ref{it:split:1} and let $\vec{d} =
  \vec{d}_{1} \gets \vec{d}_{2} \gets \dots \gets \vec{d}_{\ell}$ be
  the concatenation of transitive Hamiltonian paths $d_{i}$ in $G_{i}$
  for $1 \leq i \leq \ell$ as in \ref{it:split:2}. Then $\vec{d}_{i}$ is a Hamiltonian path
  in $G$. Since the underlying undirected graph of $G$ is complete, we
  have $\vec{d}_{i} \gets \vec{d}_{j}$ in $G$ for any vertices $d_{i}
  \in G_{i}$ and $d_{j} \in G_{j}$ in distinct strictly connected
  components with $i < j$. Thus $\vec{d}$ is also transitive and $f$
  has decomposition degree sequence $\vec{d}$ by
  \ref{it:split:1}.

  Conversely, assume \ref{it:split:2} and observe that the
  decomposition of $G$ into strictly connected components induces
  a decomposition of every transitive Hamiltonian path $\vec{d}$ in
  $G$ into Hamiltonian paths $\vec{d}_{i}$ in $G_{i}$. These are transitive,
  since transitivity is a local condition and $f$ has decomposition degree sequence $\vec{d}$ by
  \ref{it:split:2}.

  Uniqueness and thus the counting formula follow from the absence of
  equal-degree collisions in the tame case.
\end{proof}

We split the edge set $E$ of a strictly connected relation graph $G$
with vertices $V$ into its uni-directional edges $\dag{E} = \{(u,v)
\in E: (v,u) \notin E\}$ and its bi-directional edges (2-loops)
$\udg{E} = \{\{u,v\} \subseteq V \colon \{(u,v),(v,u)\} \in E\} = E
\setminus \dag{E}$.  We call the corresponding graphs on $V$ the
\emph{directed} and the \emph{undirected} subgraph of $G$, respectively.  The directed
subgraph of $G$ is a \emph{directed acyclic graph} since $G$ is the
union of transitive tournaments.  The undirected subgraph of $G$ is
connected.  It is also the union of the \emph{permutation graphs}
of $\sigma_{k}$, $1 \leq k \leq c$.

    \begin{sagesilent}
g4 = {2: [7], 3: [2, 5, 7], 5: [2, 3, 7], 7: [2, 3, 5]}

G4 = DiGraph(g4, multiedges=True, implementation='c_graph', vertex_labels=True)

new_pos4 = {}
new_pos4[2]    = (-1, 0)
new_pos4[3]    = ( 0, 1)
new_pos4[5]    = ( 0,-1)
new_pos4[7]    = ( 1, 0)
    \end{sagesilent}

\begin{figure*}
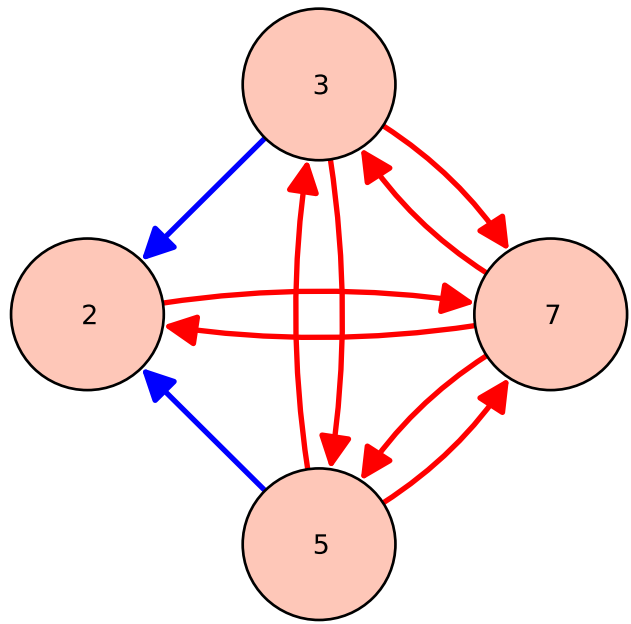

  \centering
    \sageplot[
    ]{G4.plot(
    edge_colors={'blue':[(3,2,None),(5,2,None)], 'red':[(2,7,None),(7,2,None),(3,5,None),(5,3,None),(3,7,None),(7,3,None),(5,7,None),(7,5,None)]},
    vertex_size=7000,
    pos=new_pos4,
    figsize=[3, 3])}
  \caption{The strongly connected component on 4 vertices of
    Figure~\ref{fig:conn_comp} decomposed into its undirected
    subgraph (red) and its directed subgraph (blue) with
    \textsc{Max-Sink}-sorting $7 \prec 2 \prec 5 \prec 3$.}
  \label{fig:dag_udg}
\end{figure*}

The directed subgraph $\dag{G}$ captures the requirements on the
position of the degrees in a decomposition sequence.  The undirected
subgraph $\udg{G}$ captures the admissible Ritt moves d'apr\`es
\cite{ziemue08} and thus the requirements on the shape of the
components.

Every directed acyclic graph admits a \emph{topological sorting}
$v_{1}, v_{2}, \dots, v_{\ell}$ of its vertices, where a directed edge
$v_{i} \gets v_{j}$ in $\dag{G}$ implies $i < j$, see
\citet[Section~22.4]{corlei09}. A directed acyclic graph may have
several distinct topological sortings. \cite{tar76} suggested to use
\textsc{Depth-First-Search} on $\dag{G}$. The time step, when
\textsc{Depth-First-Search} visits a vertex for the last time, is
called the \emph{finish time} of the vertex and listing the vertices
with increasing finish time yields a topological sorting. The result
is unique, if the tie-break rule for expanding in
\textsc{Depth-First-Search} is deterministic. We use the following
terminology.

Let $U(v)$ denote the open $\udg{G}$-neighborhood of a vertex $v$. It
is always nonempty. We call a vertex $v$ \emph{locally maximal}, if
its value is greater or equal than the value of every vertex in $U(v)$. Since vertices with
equal values are never connected by an edge in $\udg{G}$, a locally
maximal $v$ is always strictly greater than all vertices in
$U(v)$. Furthermore, there is at least one locally maximal vertex,
namely a ``globally'' maximal one. There is a unique enumeration of the locally maximal
vertices $d_{1}, d_{2}, \dots, d_{m}$ such that
  \begin{equation}
    \label{eq:46}
    d_{1} \gets d_{2} \gets \dots \gets d_{m}
  \end{equation}
is a directed path in $G$. Furthermore, we define for $1 \leq i \leq m$,
\begin{align}
  V_{i} & = U(d_{i}) \setminus U(d_{i+1}) \text{ and } W_{i} = V_{i} \cup
  \{d_{i}\}, \\
  V_{0} & = W_{0} = \{v \in G \colon \text{ no edge } d_{i}
  \gets v \text{ in } G \text{ for any } 1 \leq i \leq m\}, \\
  V_{m+1} & = W_{m+1} = \{v \in G \colon \text{ no edge } v \gets
  d_{i} \text{ in } G \text{
    for any } 1 \leq i \leq m\}.
\end{align}
The $W_{i}$, $0 \leq i \leq m + 1$, form a partition of all vertices
of $\dag{G}$ and we formulate the tie-break rule for
\textsc{Depth-First-Search} as follows. Given vertices $u \in W_{i}$
and $v \in W_{j}$ with $i < j$, the vertex $u$ is preferred. Given
vertices $u,v \in W_{i}$, the vertex with the larger value is
preferred. Since vertices with equal value are always connected by a
unidirectional edge in $\dag{G}$ due to \eqref{eq:13}, the search has
never to choose between to vertices with the same value and \textsc{Depth-First-Search} with this tie-break rule yields a unique
topological sorting. We call it the
\textsc{Max-Sink} topological sorting of
$\dag{G}$. \autoref{fig:dag_udg} shows the largest strongly connected component
of \autoref{fig:conn_comp} and its \textsc{Max-Sink} topological sorting.

\begin{theorem}
  \label{thm:class}
  Let $G$ be a strongly connected relation graph with
  at least two vertices, directed subgraph $\dag{G}$, and
  undirected subgraph $\udg{G}$.   Let $d_{1}, d_{2}, \dots, d_{\ell}$
  be the \textsc{Max-Sink} topological sorting of $\dag{G}$ and let $e_{i}$ be the
  product of all vertices in the open $\udg{G}$-neighborhood of
  $d_{i}$.  For every $f \in \pD{G}$ either \ref{it:class:1} or \ref{it:class:2}
  holds, and \ref{it:class:3} is also valid.
  \begin{ronumerate}
    \item (Exponential Case) \label{it:class:1} There are unique $g_{i} \in \pE{d_{i}, e_{i}}$
      for $1 \leq i \leq \ell$ and $a \in F$ such that
      \begin{equation}
        \label{eq:40}
        f = (g_{1} \circ g_{2} \circ \dots \circ g_{\ell})^{[a]}.
      \end{equation}
    \item (Trigonometric Case) \label{it:class:2} There are unique $z, a \in F$ with $z \neq 0$
      such that
      \begin{equation}
        \label{eq:41}
        f = T_{d_{1}d_{2}\cdots d_{\ell}}(x,z)^{[a]}.
      \end{equation}
    \item \label{it:class:3} If $\udg{G}$ contains no edge that connects
      two vertices both larger than 2, then the Trigonometric Case is
      included in the Exponential Case.  Otherwise, they are mutually exclusive.
  \end{ronumerate}
  Conversely,
  \begin{equation}
    \label{eq:18} \tag{3.15a}
    \pD{n, G} = \pT{d_{1}d_{2} \cdots d_{m},1}^{[F]} \cup ( \pE{d_{1},
      e_{1}} \circ \pE{d_{2}, e_{2}} \circ \dots \circ
    \pE{d_{m}, e_{m}})^{[F]}.
  \end{equation}
\end{theorem}

The $e_{i}$ are well-defined, since there are no empty neighborhoods
in the connected graph $\udg{G}$ with at least two vertices.

\begin{proof}
  We begin with the proof of existence, then show uniqueness and
  conclude with the ``converse'' \eqref{eq:18}.

  The \textsc{Max-Sink} topological sorting $d_{1}, d_{2}, \dots,
  d_{\ell}$ of $\dag{G}$ yields a transitive Hamiltonian path
  \begin{equation}
    \label{eq:47}
    \vec{d} = d_{1} \prec d_{2} \prec \dots \prec d_{\ell}
  \end{equation}
  in $G$. For the rest of the proof, we identify the \textsc{max-sink}
  topological sorting with the corresponding transitive Hamiltonian path.

  We (re)label the locally maximally vertices $d_{1},
  d_{2}, \dots, d_{m}$ and the elements of $V_{i}$ as $d_{i}^{(1)},
  d_{i}^{(2)}, \dots, d_{i}^{(\ell_{i})}$ for $0 \leq i \leq m+1$ and
  $\ell_{i} = \# V_{i}$ such that
  \begin{align}
    \vec{d} & = (d_{0}^{(1)}, \dots, d_{0}^{(\ell_{0})}, d_{1}, d_{1}^{(1)}, \dots,
    d_{1}^{(\ell_{1})}, d_{2}, d_{2}^{(1)}, \dots, \\
    & \quad d_{m-1}^{(\ell_{m-1})},
    d_{m}, d_{m}^{(1)}, \dots, d_{m}^{(\ell_{m})}, d_{m+1}^{(1)},
    \dots, d_{m+1}^{(\ell_{m+1})}) \\
    & = (V_{0}, d_{1}, V_{1}, d_{2}, \dots, d_{m}, V_{m}, V_{m+1}),
  \end{align}
  where the $V_{i}$ are read as tuple $(d_{i}^{(1)},
  d_{i}^{(2)}, \dots, d_{i}^{(\ell_{i})})$.
  Then $f$ has a decomposition
  \begin{align}
    \label{eq:90} \tag{3.15b}
    f & = G_{0}^{(1)} \circ \dots \circ G_{0}^{(\ell_{0})} \circ G_{1}
    \circ G_{1}^{(1)} \circ \dots \circ G_{1}^{(\ell_{1})} \circ G_{2}
    \circ G_{2}^{(1)} \circ \dots \\
    & \quad \quad \circ G_{m-1}^{(\ell_{m-1})} \circ
    G_{m} \circ G_{m}^{(1)} \circ \dots \circ G_{m}^{(\ell_{m})} \circ
    G_{m+1}^{(1)} \circ \dots \circ G_{m+1}^{(\ell_{m+1})}
  \end{align}
  with $G_{i}^{(j)} \in \pP{d_{i}^{(j)}}$ for $0 \leq i \leq m+1$, $1
  \leq j \leq \ell_{i}$, and $G_{i} \in \pP{d_{i}}$ for $1 \leq i \leq m$.

  We assume for the moment that all edges in $\udg{G}$ contain a
  $2$. Then \autoref{thm:Ritt2} reduces to the exponential case, and
  we proceed as follows. First, we show that every $G_{i}^{(j)}$ for
  $1 \leq i \leq m$, $1 \leq j \leq \ell_{i}$, is of the form
  $g_{i}^{[a_{i}]}$ for unique $g_{i} \in \pE{d_{i}^{(j)}, e_{i}^{(j)}}$ and unique
  $a_{i} \in F$. Then, we extend this to $i = 0$ and $i =
  m+1$. Finally, we show that the shifting parameters $a_{i}$ are
  ``compatible'' such that a single shifting parameter $a$ suffices.

  For every $1 \leq i \leq m$, we use \textsc{Bubble-Sort}
  \autoref{algo:bubble} with \autoref{lem:1} to obtain the
  decomposition degree sequence
  \begin{equation}
    (V_{0}, d_{1}, V_{1}, \dots, d_{i}, V_{i}, d_{i}^{(\ell_{i+1})}, \dots
    d_{i}^{(m_{i})}, d_{i+1}, \hat{V}_{i+1}, \dots, d_{m}, \hat{V}_{m}, V_{m+1}),
  \end{equation}
  where $U(d_{i}) = V_{i} \cup \{d_{i}^{(\ell_{i+1})}, \dots
    d_{i}^{(m_{i})}\}$ and the latter elements have been omitted from
    $V_{i+1}, \dots, V_{m}$. We have $e_{i} = \prod_{1 \leq j \leq
      m_{i}} d_{i}^{(j)}$ and this implies the two decomposition
    degree sequences
  \begin{gather}
    (V_{0}, d_{1}, V_{1}, \dots, d_{i}, e_{i}, d_{i+1}, \hat{V}_{i+1}, \dots, d_{m}, \hat{V}_{m}, V_{m+1}),
    (V_{0}, d_{1}, V_{1}, \dots, e_{i}, d_{i}, d_{i+1}, \hat{V}_{i+1}, \dots, d_{m}, \hat{V}_{m}, V_{m+1}).
  \end{gather}
  Thus, there are unique $g_{i} \in \pE{d_{i}, e_{i}}$ and $a_{i} \in F$
  such that in \eqref{eq:90}, we have
  \begin{equation}
    \label{eq:49}
    G_{i} \circ G_{i}^{(1)} \circ \dots \circ G_{i}^{(\ell_{i})} =
    (g_{i} \circ x^{d_{i}^{(1)}} \circ \dots \circ x^{(d_{i}^{(\ell_{i})})})^{[a_{i}]}.
  \end{equation}
  The same form applies to $i = 0$, since there is at
  least one $d_{i}^{(j)}$ with $1 \leq i \leq m$, $1 \leq j \leq
  \ell_{i}$ that is in the $\udg{G}$-neighborhood of some element of
  $V_{0}$ due to the strong connectedness of $G$. And since there is
  no locally maximal element in $V_{0}$ all components are of the form
  $x^{d_{0}^{(j)}}$ with possible some shift applied. Every connection
  in $G$ relates the corresponding shifting parameters and since $G$
  has a Hamiltonian path, they are all determined by a single choice.

  Now, for the general case, where some collisions may be
  trigonometric, but not exponential. For any two
  locally maximal vertices $d_{i}$ and $d_{j}$ there is some vertex $d
  \in U(d_{i}) \cap U(d_{j})$. This shows, that either all blocks fall
  into the exponential case or all blocks fall into the trigonometric
  case. The two cases are disjoint if and only if there is some edge
  in $\udg{G}$ that connects two vertices both with value greater than
  $2$.

The stabilizer of original shifting is
$\{0\}$ for nonlinear monic original polynomials and there are no
equal-degree collisions. Hence the representation is unique.

The converse \eqref{eq:18} is a direct computation.
\end{proof}

\begin{lemma}
  \label{lem:1}
  Let $i < j$ and $d_{j}$ in the open $\udg{G}$-neighborhood of
  $d_{i}$. Then for every $i < k < j$, we have $d_{k}$ in the open
  $\udg{G}$-neighborhood of $d_{i}$ or $d_{j}$ or both.
\end{lemma}

\begin{proof}
  The tournaments underlying $G$ are acyclic. Therefore, if $d_{i}
  \prec d_{k} \prec d_{j}$ and $d_{j} \prec d_{k}$, then at least one
  other edge is bidirectional, too.
\end{proof}

\section{Exact Counting of Decomposable Polynomials}
\label{sec:counting}

The classification of \autoref{thm:class} yields the exact number of decomposable polynomials at
degree $n$ over a finite field $\Fq$.

\begin{theorem}
  \label{thm:count} Let $G$ be a strongly connected relation graph
  with undirected subgraph $\udg{G}$.  Let $d_{1}, d_{2}, \dots,
  d_{\ell}$ be the vertices of $\udg{G}$ and $e_{i}$ be the product of
  all vertices in the (open) $\udg{G}$-neighborhood of $d_{i}$.  Let
  $\delta_{\udg{G},2}$ be $1$ if there is no edge in $\udg{G}$ between
  two vertices both larger than $2$ and let $\delta_{\udg{G},2}$ be
  $0$ otherwise.  Then
  \begin{equation}
    \label{eq:26}
    \# \pD{G} = \begin{cases*}
      q^{d-1} & if $G = \{d\}$, \\
      q \cdot (\prod_{d_{i} \in G} q^{\floor{d_{i}/e_{i}}} +
      (1- \delta_{\udg{G},2}) \cdot (q-1)) & otherwise.
    \end{cases*}
  \end{equation}
\end{theorem}

\begin{proof}
  For $G = \{d\}$, this follows from \eqref{eq:42}. Otherwise from the
  (non)uniqueness of the parameters in
  \autoref{thm:class}.
\end{proof}

We are finally ready to employ the inclusion-exclusion formula
\eqref{eq:9} from the beginning.  For a nonempty set $D$ of nontrivial
divisors of $n$, it requires $\# \pD{n, D} = \# \pD{n, \vec{D}}$ for
$\vec{D} = \{(d, n/d) \colon d \in D\}$.  We compute the normalization
$\vec{D}^{*}$ by repeated application of \autoref{algo:refine} and
derive the relation graph of $\vec{D}^{*}$.  Then $\# \pD{n, \vec{D}}
= \# \pP{G}$ and the latter follows from \autoref{thm:split} and
\autoref{thm:count}.

\begin{algorithm2e}
\caption{\textsc{Count Decomposables}}
\label{algo:count}

\KwIn{positive integer $n$}
\KwOut{$\# \pDa{n}(\Fq)$ as a polynomial in $q$ for $n$ coprime to
  $q$}

\If{$n=1$ or $n$ is prime}{
\KwRet{$0$}
}
total $\gets 0$\;
$N \gets \{1 < d < n\colon d \mid n \}$\;
\For{$\emptyset \neq D \subseteq N$}{
$\vec{D} \gets \{(d, n/d) \colon d \in D\}$\;
$\vec{D}^{*} \gets \text{normalization of } \vec{D}$\;
$G \gets \text{relation graph of } \vec{D^{^{*}}}$\;
collisions $\gets 1$\;
  \For{strongly connected components $G_{j}$ of $G$}{
  $\udg{G}_{j} \gets \text{undirected subgraph of } G_{j}$\;
  \eIf{$G_{j} = \{ d \}$}{
  connected $\gets q^{d}$\;
  }{
  $\{d_{1}, d_{2}, \dots, d_{\ell}\} \gets G_{j}$\;
  \For{$i = 1, \dots, \ell$}{
    $U \gets \text{open neighborhood of } d_{i} \text{ in } \udg{G}_{j}$\;
    $e_{i} \gets \prod_{v \in U}$
  }
  connected $\gets \prod_{d_{i} \in G_{j}}
  q^{\floor{d_{i}/e_{i}}}$\;
  \If{some edge in $\udg{G}_{j}$ connects two vertices both larger than
    $2$}
{
connected $\gets \text{connected} + q-1$\;
}
  connected $\gets \text{connected} \cdot q$\;
  }
collisions $\gets \text{collisions} \cdot \text{connected}$\;
}
$k \gets \# D$\;
total $\gets \text{total} + (-1)^{k} \text{collisions}$\;
}
\KwRet{total}
\end{algorithm2e}

This is easy to implement, see \autoref{algo:count}, and yields the
exact expressions for $\# \pDa{n} (\Fq)$ at lightning speed, see
\autoref{tab:exact}.  Where no exact expression was previously known,
we compare this to the upper and lower bounds of \cite{gat08c}.

\begin{table*}
\centering
\caption{Exact values of $\# \pDa{n} (\Fq)$ in the tame case for
  composite $n \leq 50$, consistent with the upper and lower bounds
  (in the last column) or exact values (no entry in the last column) of \citet[Theorem~5.2]{gat08c}.}
\label{tab:exact}
\end{table*}

\section{Conclusion}
\label{sec:conclusion}

We presented a normal form for multi-collisions of decompositions of
arbitrary length with exact description of the (non)uniqueness of the
parameters.  This lead to an efficiently computable formula for the
exact number of such collisions at degree $n$ over a finite field of
characteristic coprime to $p$.  We concluded with an algorithm to
compute the exact number of decomposable polynomials at degree $n$
over a finite field $\Fq$ in the tame case.

We introduced the relation graph of a set of collisions which may be
of independent interest due to its connection to permutation graphs.
It would be interesting to characterize sets $\vec{D}$ of ordered
factorizations that lead to identical contributions $\# \pD{n,
  \vec{D}}$ and to quickly derive $\# \pD{n, \vec{D} \cup \{\vec{e}\}}$
form $\# \pD{n, \vec{D}}$ or conversely.  Finally, this
work deals with polynomials only and the study of rational functions
with the same methods remains open.

\section{Acknowledgements}
\label{sec:ack}

Many thanks go to Joachim von zur Gathen for useful discussions and pointers to the
literature.  This work was funded by the B-IT Foundation and the Land Nordrhein-
Westfalen.

\nocite{grakal03} 

\bibliography{journals,refs,lncs}

\providecommand{\ymd}[3]{\csname @ifempty\endcsname{#1}{?#1/#2/#3?}{\csname
  @ifempty\endcsname{#2}{?#1/#2/#3?}{\csname
  @ifempty\endcsname{#3}{?#1/#2/#3?}{{\relax \day=#3\relax \month=#2\relax
  \year=#1\relax \number\day~\ifcase\month\or January\or February\or March\or
  April\or May\or June\or July\or August\or September\or October\or November\or
  December\fi \space\ifnum\year>0\relax \number\year \else \csname
  count@\endcsname1\relax \expandafter\advance\csname count@\endcsname-\year
  \expandafter\number\csname count@\endcsname~BC\fi}}}}}
  \providecommand{\hide}[1]{.} \providecommand{\Hide}[1]{\unskip}
  \providecommand{\gobble}[1]{} \providecommand{\todo}[1]{\textbf{`?`?`?#1???}}
  \providecommand{\Name}[1]{#1} \providecommand{\Textgreek}[1]{\textgreek{#1}}
  \providecommand{\at}{\char64\relax}
  \providecommand{\cyr}{\PackageError{cyrillic}{Package not loaded. Use
  \string\usepackage{cyrillic} to define \string\cyr\space appropriately}{}
  \gdef\cyr{\def\cprime{c'}{\bf ?cyr?}}\cyr}
  \makeatletter\protected@write\@auxout{}{\string
  \gdef\string\abbr{\string\csname\space
  @gobble\string\endcsname}}\gdef\abbr{}\makeatother
  \makeatletter\protected@write\@auxout{}{\string
  \gdef\string\bibliographyonly{\string\protect\string\abbr}}\gdef\bibliographyonly{\protect\abbr}\makeatother
  \makeatletter\protected@write\@auxout{}{\string
  \gdef\string\bodyonly{}}\gdef\bodyonly#1{}\makeatother
\begin{thebibliography}{33}
\providecommand{\natexlab}[1]{#1}
\providecommand{\url}[1]{\texttt{#1}}
\providecommand{\urlprefix}{URL }
\expandafter\ifx\csname urlstyle\endcsname\relax
  \providecommand{\doi}[1]{doi:\discretionary{}{}{}#1}\else
  \providecommand{\doi}{doi:\discretionary{}{}{}\begingroup
  \urlstyle{rm}\Url}\fi
\providecommand{\selectlanguage}[1]{\relax}

\bibitem[{Avanzi \& Zannier(2003)}]{avazan03}
\textsc{Roberto~M. Avanzi} \& \textsc{Umberto~M. Zannier} (2003).
\newblock The equation {$f(X)=f(Y)$} in rational functions {$X=X(t)$},
  {$Y=Y(t)$}.
\newblock \emph{Compositio Math.} \textbf{139}(3), 263--295.
\newblock \doi{10.1023/B:COMP.0000018136.23898.65}.

\bibitem[{Bach \& Shallit(1997)}]{bacsha97}
\textsc{Eric Bach} \& \textsc{Jeffrey Shallit} (1997).
\newblock \emph{Algorithmic Number Theory, Vol.1: Efficient Algorithms}.
\newblock MIT~Press, Cambridge~MA, second printing edition.
\newblock ISBN 0-262-02405-5.

\bibitem[{Barton \& Zippel(1985)}]{barzip85}
\textsc{David~R. Barton} \& \textsc{Richard Zippel} (1985).
\newblock {Polynomial} {Decomposition} {Algorithms}.
\newblock \emph{Journal of Symbolic Computation} \textbf{1}, 159--168.

\bibitem[{Blankertz, von~zur Gathen \& Ziegler(2013)}]{blagat13}
\textsc{Raoul Blankertz}, \textsc{Joachim von~zur Gathen} \& \textsc{Konstantin
  Ziegler} (2013).
\newblock Compositions and collisions at degree $p^2$.
\newblock \emph{Journal of Symbolic Computation} \textbf{59}, 113--145.
\newblock ISSN 0747-7171.
\newblock \urlprefix\url{http://dx.doi.org/10.1016/j.jsc.2013.06.001}.
\newblock Also available at \url{http://arxiv.org/abs/1202.5810}. Extended
  abstract in \bgroup\em Proceedings of the 2012 International Symposium on
  Symbolic and Algebraic Computation ISSAC~'12, {\rm Grenoble, France}\egroup\
  (2012), 91--98.

\bibitem[{Bodin, D{\`{e}}bes \& Najib(2009)}]{boddeb09}
\textsc{Arnaud Bodin}, \textsc{Pierre D{\`{e}}bes} \& \textsc{Salah Najib}
  (2009).
\newblock Indecomposable polynomials and their spectrum.
\newblock \emph{Acta Arithmetica} \textbf{139(1)}, 79--100.

\bibitem[{Cade(1985)}]{cad85}
\textsc{John~J. Cade} (1985).
\newblock A {New} {Public-key} {Cipher} {Which} {Allows} {Signatures}.
\newblock In \emph{Proceedings of the 2nd SIAM Conference on Applied Linear
  Algebra}. SIAM, Raleigh NC.

\bibitem[{Cormen \emph{et~al.}(2009)Cormen, Leiserson, Rivest \&
  Stein}]{corlei09}
\textsc{Thomas~H. Cormen}, \textsc{Charles~E. Leiserson}, \textsc{Ronald~L.
  Rivest} \& \textsc{Clifford Stein} (2009).
\newblock \emph{Introduction to Algorithms}.
\newblock MIT~Press, Cambridge~MA, London~UK, 3rd edition.
\newblock ISBN 978-0-262-03384-8 (hardcover), 978-0-262-53305-8 (paperback),
  1312 pages .

\bibitem[{Dorey \& Whaples(1974)}]{dorwha74}
\textsc{F.~Dorey} \& \textsc{G.~Whaples} (1974).
\newblock {Prime} and {Composite} {Polynomials}.
\newblock \emph{Journal of Algebra} \textbf{28}, 88--101.
\newblock \urlprefix\url{http://dx.doi.org/10.1016/0021-8693(74)90023-4}.

\bibitem[{Engstrom(1941)}]{eng41}
\textsc{H.~T. Engstrom} (1941).
\newblock {Polynomial} {Substitutions}.
\newblock \emph{American Journal of Mathematics} \textbf{63}, 249--255.
\newblock \urlprefix\url{http://www.jstor.org/stable/pdfplus/2371520.pdf}.

\bibitem[{Fried \& MacRae(1969)}]{frimac69}
\textsc{Michael~D. Fried} \& \textsc{R.~E. MacRae} (1969).
\newblock On the invariance of chains of Fields.
\newblock \emph{Illinois Journal of Mathematics} \textbf{13}, 165--171.

\bibitem[{von~zur Gathen(1990{\natexlab{a}})}]{gat90c}
\textsc{Joachim von~zur Gathen} (1990{\natexlab{a}}).
\newblock {Functional} {Decomposition} of {Polynomials}: the {Tame} {Case}.
\newblock \emph{Journal of Symbolic Computation} \textbf{9}, 281--299.
\newblock \urlprefix\url{http://dx.doi.org/10.1016/S0747-7171(08)80014-4}.

\bibitem[{von~zur Gathen(1990{\natexlab{b}})}]{gat90d}
\textsc{Joachim von~zur Gathen} (1990{\natexlab{b}}).
\newblock {Functional} {Decomposition} of {Polynomials}: the {Wild} {Case}.
\newblock \emph{Journal of Symbolic Computation} \textbf{10}, 437--452.
\newblock \urlprefix\url{http://dx.doi.org/10.1016/S0747-7171(08)80054-5}.

\bibitem[{von~zur Gathen(2002)}]{gat02c}
\textsc{Joachim von~zur Gathen} (2002).
\newblock Factorization and Decomposition of Polynomials.
\newblock In \emph{The Concise Handbook of Algebra}, edited by
  \textsc{Alexander~V. Mikhalev} \& \textsc{G{\"{u}}nter~F. Pilz}, 159--161.
  Kluwer Academic Publishers.
\newblock ISBN 0-7923-7072-4.

\bibitem[{von~zur Gathen(2014{\natexlab{a}})}]{gat08c}
\textsc{Joachim von~zur Gathen} (2014{\natexlab{a}}).
\newblock Counting decomposable univariate polynomials.
\newblock \emph{To appear in Combinatorics, Probability and Computing, Special
  Issue} \textbf{\unskip\null}.
\newblock Extended abstract in \bgroup\em Proceedings of the 2009 International
  Symposium on Symbolic and Algebraic Computation ISSAC~'09, {\rm Seoul,
  Korea}\egroup\ (2009). Preprint (2008) available at
  \url{http://arxiv.org/abs/0901.0054}.

\bibitem[{von~zur Gathen(2014{\natexlab{b}})}]{gat12a}
\textsc{Joachim von~zur Gathen} (2014{\natexlab{b}}).
\newblock Normal form for Ritt's Second Theorem.
\newblock \emph{Finite Fields and Their Applications} \textbf{27}, 41--71.
\newblock ISSN 1071-5797.
\newblock \urlprefix\url{http://dx.doi.org/10.1016/j.ffa.2013.12.004}.
\newblock Also available at \url{http://arxiv.org/abs/1308.1135}.

\bibitem[{von~zur Gathen, Kozen \& Landau(1987)}]{gatkoz87}
\textsc{Joachim von~zur Gathen}, \textsc{Dexter Kozen} \& \textsc{Susan Landau}
  (1987).
\newblock Functional Decomposition of Polynomials.
\newblock In \emph{Proceedings of the 28th Annual IEEE Symposium on Foundations
  of Computer Science, {\rm Los~Angeles~CA}}, 127--131. IEEE Computer Society
  Press, Washington~DC.
\newblock \urlprefix\url{http://dx.doi.org/10.1109/SFCS.1987.29}.

\bibitem[{Giesbrecht(1988)}]{gie88b}
\textsc{Mark~William Giesbrecht} (1988).
\newblock \emph{{Some} {Results} on the {Functional} {Decomposition} of
  {Polynomials}}.
\newblock Master's thesis, Department of Computer Science, University of
  Toronto.
\newblock Technical Report 209/88. Available as
  \url{http://arxiv.org/abs/1004.5433}.

\bibitem[{Grabmeier, Kaltofen \& Weispfenning(2003)}]{grakal03}
\textsc{Johannes Grabmeier}, \textsc{Erich Kaltofen} \& \textsc{Volker
  Weispfenning} (editors) (2003).
\newblock \emph{Computer Algebra Handbook -- Foundations, Applications,
  Systems}.
\newblock Springer-Verlag, Berlin, Heidelberg, New York.
\newblock ISBN 3-540-65466-6.
\newblock \urlprefix\url{http://www.springer.com/978-3-540-65466-7}.

\bibitem[{Gutierrez \& Kozen(2003)}]{gutkoz03}
\textsc{Jaime Gutierrez} \& \textsc{Dexter Kozen} (2003).
\newblock Polynomial Decomposition.
\newblock In  \cite{grakal03}, section 2.2.4 (pages 26--28).
\newblock \urlprefix\url{http://www.springer.com/978-3-540-65466-7}.

\bibitem[{Gutierrez \& Sevilla(2006)}]{gutsev06}
\textsc{Jaime Gutierrez} \& \textsc{David Sevilla} (2006).
\newblock On Ritt's decomposition theorem in the case of finite fields.
\newblock \emph{Finite Fields and Their Applications} \textbf{12}(3), 403--412.
\newblock \urlprefix\url{http://dx.doi.org/10.1016/j.ffa.2005.08.004}.

\bibitem[{Kozen \& Landau(1989)}]{kozlan89}
\textsc{Dexter Kozen} \& \textsc{Susan Landau} (1989).
\newblock Polynomial Decomposition Algorithms.
\newblock \emph{Journal of Symbolic Computation} \textbf{7}, 445--456.
\newblock \urlprefix\url{http://dx.doi.org/10.1016/S0747-7171(89)80027-6}.
\newblock An earlier version was published as Technical Report 86-773, Cornell
  University, Department of Computer Science, Ithaca, New York, 1986.

\bibitem[{Kozen, Landau \& Zippel(1996)}]{kozlan96}
\textsc{Dexter Kozen}, \textsc{Susan Landau} \& \textsc{Richard Zippel} (1996).
\newblock Decomposition of Algebraic Functions.
\newblock \emph{Journal of Symbolic Computation} \textbf{22}, 235--246.

\bibitem[{Levi(1942)}]{lev42}
\textsc{H.~Levi} (1942).
\newblock Composite {Polynomials} with coefficients in an arbitrary {Field} of
  characteristic zero.
\newblock \emph{American Journal of Mathematics} \textbf{64}, 389--400.

\bibitem[{Medvedev \& Scanlon(2014)}]{medsca14}
\textsc{Alice Medvedev} \& \textsc{Thomas Scanlon} (2014).
\newblock Invariant varieties for polynomial dynamical systems.
\newblock \emph{Annals of Mathematics} \textbf{179}(1), 81--177.
\newblock \urlprefix\url{http://dx.doi.org/10.4007/annals.2014.179.1.2}.
\newblock Also available at \url{http://arxiv.org/abs/0901.2352v3}.

\bibitem[{Ritt(1922)}]{rit22}
\textsc{J.~F. Ritt} (1922).
\newblock Prime and Composite Polynomials.
\newblock \emph{Transactions of the American Mathematical Society} \textbf{23},
  51--66.
\newblock \urlprefix\url{http://www.jstor.org/stable/1988911}.

\bibitem[{Schinzel(1982)}]{sch82c}
\textsc{Andrzej Schinzel} (1982).
\newblock \emph{Selected Topics on Polynomials}.
\newblock Ann Arbor; The University of Michigan Press.
\newblock ISBN 0-472-08026-1.

\bibitem[{Schinzel(2000)}]{sch00c}
\textsc{Andrzej Schinzel} (2000).
\newblock \emph{Polynomials with special regard to reducibility}.
\newblock Cambridge University Press, Cambridge, UK.
\newblock ISBN 0521662257.

\bibitem[{Tarjan(1976)}]{tar76}
\textsc{Robert~Endre Tarjan} (1976).
\newblock Edge-Disjoint Spanning Trees and Depth-First Search.
\newblock \emph{Acta Informatica} \textbf{6}, 171--185.
\newblock \urlprefix\url{http://dx.doi.org/10.1007/BF00268499}.

\bibitem[{Tortrat(1988)}]{tor88a}
\textsc{Pierre Tortrat} (1988).
\newblock Sur la composition des polyn{\^{o}}mes.
\newblock \emph{Colloquium Mathematicum} \textbf{55}(2), 329--353.

\bibitem[{Turnwald(1995)}]{tur95}
\textsc{Gerhard Turnwald} (1995).
\newblock On Schur's Conjecture.
\newblock \emph{Journal of the Australian Mathematical Society, Series~A}
  \textbf{58}, 312--357.
\newblock
  \urlprefix\url{http://anziamj.austms.org.au/JAMSA/V58/Part3/Turnwald.html}.

\bibitem[{Zannier(1993)}]{zan93}
\textsc{U{\hide{mberto}}~Zannier} (1993).
\newblock {Ritt's} {Second} {Theorem} in arbitrary characteristic.
\newblock \emph{Journal f{\"{u}}r die reine und angewandte Mathematik}
  \textbf{445}, 175--203.

\bibitem[{Zannier(2008)}]{zan08}
\textsc{Umberto Zannier} (2008).
\newblock On composite lacunary polynomials and the proof of a conjecture of
  Schinzel.
\newblock \emph{Inventiones mathematicae} \textbf{174}, 127--138.
\newblock ISSN 0020-9910 (Print) 1432-1297 (Online).
\newblock \urlprefix\url{http://dx.doi.org/10.1007/s00222-008-0136-8}.

\bibitem[{Zieve \& M{\"{u}}ller(2008)}]{ziemue08}
\textsc{Michael~E. Zieve} \& \textsc{Peter M{\"{u}}ller} (2008).
\newblock On Ritt's Polynomial Decomposition Theorems.
\newblock \emph{Submitted,} \urlprefix\url{http://arxiv.org/abs/0807.3578}.

\end{thebibliography}
\bibliographystyle{cc2e}

\end{document}